\documentclass{amsart}

\usepackage{amsfonts, amsthm, amssymb, amsmath, stmaryrd}
\usepackage{mathrsfs,array}
\usepackage{calligra}
\usepackage{eucal,color,times,enumerate,accents}
\usepackage[all]{xy}
\usepackage{mathtools}
\usepackage{url}
\usepackage{enumitem}
\usepackage{enumerate}
\usepackage[pdftex,bookmarksnumbered,bookmarksopen]
{hyperref}
\usepackage{lipsum}
\usepackage{adjustbox}
\usepackage{cite}
\usepackage[colorinlistoftodos]{todonotes}
\usepackage{tikz-cd}
\usetikzlibrary{decorations.pathreplacing}
\usetikzlibrary{positioning}
\usepackage{dsfont}
\usepackage{tikz}
\usetikzlibrary{matrix,arrows,decorations.pathmorphing,positioning}
\usepackage{calligra}
\usepackage[colorinlistoftodos]{todonotes}
\usepackage{xy}
\xyoption{all}
\usepackage{tikz}
\usetikzlibrary{matrix,arrows,decorations.pathmorphing}
\usepackage{leftidx}
\input xy
\xyoption{all}

\usepackage{tikz-cd}
\theoremstyle{plain}
\newtheorem{thm}{Theorem}[section]

\newtheorem{lemma}[thm]{Lemma}
\newtheorem*{thm*}{Theorem}
\newtheorem*{prop*}{Proposition}
\newtheorem{prop}[thm]{Proposition}
\newtheorem{cor}[thm]{Corollary}
\theoremstyle{definition}
\newtheorem{defi}[thm]{Definition}

\theoremstyle{remark}
\newtheorem{rmk}{Remark}
\newtheorem{rmks}{Remarks}

\newtheorem*{exs}{Examples}
\newtheorem{claim}{Claim}[thm]
\newtheorem*{claimproof}{Proof of claim}

\numberwithin{equation}{subsection}

\newcommand{\overbar}[1]{\mkern 1.5mu\overline{\mkern-1.5mu#1\mkern-1.5mu}\mkern 1.5mu}
\DeclareMathOperator{\lcm}{lcm}

\DeclareMathOperator{\Br}{Br}

\DeclareMathOperator{\End}{End}

\DeclareMathOperator{\Pic}{Pic}
\DeclareMathOperator{\Hom}{Hom}
\DeclareMathOperator{\Gal}{Gal}

\DeclareMathOperator{\Cl}{Cl}

\DeclareMathOperator{\Res}{Res}
\DeclareMathOperator{\inv}{inv}

\DeclareMathOperator{\im}{Im}

\DeclareMathOperator{\disc}{disc}

\DeclareMathOperator{\art}{art}

\DeclareMathOperator{\tr}{tr}

\DeclareMathOperator{\Gl}{GL}

\newcommand{\id}{\operatorname{Id}}

\newcommand{\res}{\operatorname{res}}

\newcommand{\ab}{\operatorname{ab}}
\newcommand{\cl}{\operatorname{Cl}}

\newcommand{\Hdg}{\operatorname{Hdg}}
\newcommand{\Aut}{\operatorname{Aut}}
\newcommand{\NS}{\operatorname{NS}}

\newcommand{\Z}{\mathbb{Z}}
\newcommand{\Q}{\mathbb{Q}}

\newcommand{\R}{\mathbb{R}}
\newcommand{\Pp}{\mathbb{P}}
\newcommand{\A}{\mathbb{A}}

\newcommand{\Oo}{\mathcal{O}}

\newcommand{\D}{\mathcal{D}}

\newcommand{\F}{\mathcal{F}}

\newcommand{\et}{\textrm{\'{e}t}}

\newcommand{\C}{\mathbb{C}}

\newcommand{\ord}{\mathrm{ord}}

\newcommand{\adjunction}[4]{\xymatrix@1{#1{\ } \ar@<-0.3ex>[r]_{ {\scriptstyle #2}} & {\ } #3 \ar@<-0.3ex>[l]_{ {\scriptstyle #4}}}}

 \usepackage[foot]{amsaddr}

\begin{document}

\title{Fields of definition of K3 surfaces with complex multiplication}

\author{Domenico Valloni}
\address{Leibniz Universität Hannover, Welfengarten 1, 30167 Hannover}
\email{valloni@math.uni-hannover.de}
\maketitle

\begin{abstract}
Let $X/\C$ be a K3 surface with complex multiplication by the ring of integers of a CM field $E$. We show that $X$ can always be defined over an Abelian extension $K/E$ explicitly determined by the discriminant form of the lattice $\NS(X)$. We then construct a model of $X$ over $K$ via Galois-descent and we study some of its basic properties, in particular we determine its Galois representation explicitly. Finally, we apply our results to give upper and lower bounds for a minimal field of definition for $X$ in terms of the class number of $E$ and the discriminant of $\NS(X)$. 
\end{abstract}

\tableofcontents

\section{Introduction}
An algebraic K3 surface $X/\C$ is said to have complex multiplication if its transcendental lattice $T(X)$ admits as many Hodge endomorphisms as possible: 
\begin{enumerate}
\item $E \coloneqq \End_{\Hdg}(T(X)_\Q)$ is a CM field and
\item $\dim_E T(X)_\Q = 1$,
\end{enumerate}
where $T(X)_\Q \coloneqq T(X) \otimes_\Z \Q$ and $\End_{\Hdg}(T(X)_\Q)$ are the endomorphisms of $T(X)_\Q$ that respect the Hodge decomposition. Note that in fact (2) implies (1); moreover, from (2) it also follows that $[E \colon \Q] \leq 20$ necessarily. 

It is known, see \cite[Theorem 3]{taelman2016}, that for any CM field $E$ with $[E \colon \Q ] \leq 20$ there exist K3 surfaces with complex multiplication by $E$. On the other hand, these surfaces are very difficult to construct geometrically, and in fact the following list exhausts all the known examples at the present moment: K3 surfaces of maximal Picard rank (related geometrically to CM elliptic curves via their Shioda-Inose structure), Kummer surfaces associated to CM Abelian surfaces, or K3 surfaces that are dominated by a product of two curves with CM Jacobians. Other examples are given by K3 surfaces $X$ for which the field $E$ is generated by the action of $\Aut(X)$ on the transcendental part of the cohomology. Note that in this latter case, the CM field $E$ is always a cyclotomic field (see \cite{kondo1992automorphisms} and \cite{SCHUTT2010206} for explicit constructions). 

Piateski-Shapiro and Shafarevich showed that CM K3 surfaces can always be defined over $\overline{\Q}$. In this work, we shall determine some preferred fields of definition for CM K3 surfaces, which will be abelian extensions of the CM field $E$. These extensions are determined by the arithmetic of $E$ (e.g., its relative class group) as well as the structure of $T(X)$ as an integral quadratic form. 

Among other things, our results allow us to give precise asymptotics for the minimal degree of a field of definition of $X$ in terms of the discriminant of its N\'{e}ron-Severi group and the relative class number of $E$ (Theorem \ref{theorem asymptotics introduction}), and to prove the finiteness result of Orr and Skorobogatov \cite{MR3830546} by techniques that are independent from the theory of Abelian varieties, but at the cost of restricting to K3 surfaces whose order of complex multiplication is the maximal one. Before stating our results, we recall some facts and notation. 

\subsubsection{K3 surfaces with complex multiplication} \label{subsection introduction K3 surfaces with CM}
If $X/ \C$ is a K3 surfaces with complex multiplication, then the CM field $E \coloneqq \End_{\Hdg} (T(X)_\Q)$ can be naturally embedded into $\C$ via the map 
$$E \hookrightarrow \End(T^{1, -1}(X)) \cong \C.$$ 

The ring $\Oo_X := \End_{\Hdg} (T(X))$ of integral Hodge endomorphisms is an order in $E$, and one says that $X$ is \textit{principal} if $\Oo_X$ is the ring of integers of $E$. As shown in \cite[Proposition 6.11]{VALLONI2021107772}, given a CM number field $E$ with $[E \colon \Q] \leq 10$, there are infinitely many K3 surfaces with CM by $\Oo_E$. This is in sharp contrast with the theory of CM Abelian varieties, for which the analogous set is finite.  

By work of Mukai, Nikulin and Buskin (\cite{Mukai}, \cite{NikulinHodgeK3} and \cite{Buskin}) the Hodge conjecture holds for $X \times X$ when $X$ is a CM K3 surface, meaning that all the elements of $\End_{\Hdg} (T(X)) \subset \mathrm{H}_B^4(X \times X, \Q)(2)$ are algebraic.

\begin{defi}
If $X$ is defined over a subfield $L \subset \C$, one says that $X$ has CM over $L$ if all the elements of $\End_{\Hdg}(T(X)_\Q)$ are defined over $L$ as \textit{cohomology classes}.
\end{defi}
The finite Abelian extensions of $E$ that will provide our fields of definition were introduced in Section 9 of \cite{VALLONI2021107772}. There, it is explained how to associate to every ideal $I \subset \Oo_E$ a finite Abelian extension $F_{I}(E) / E$ via class field theory. The various $F_I(E)$ are to K3 surfaces with complex multiplication what ray class fields are to CM elliptic curves, and this analogy is also supported by the way we have employed them to study the Brauer groups of CM K3 surfaces. For convenience of the reader, in Section \ref{K3 class filds subsection} we list the properties of such extensions that are needed in this work. We refer the reader to Section 9 and 10 of \cite{VALLONI2021107772} for a more complete treatment. 
\subsection{Our results} \label{subsection our results}
Given a CM, principal K3 surface $X / \C$ one has its discriminant ideal $\mathcal{D}_X \subset \Oo_E$, see Definition \ref{defi discriminant ideal}. There is an isomorphism of $\Oo_E-$modules $$\Oo_E / \D_X \cong T^\vee(X) / T(X)$$ 
where $T^\vee(X)$ denotes the dual of $T(X)$ (Proposition \ref{discriminantP}). In particular $$\mathrm{Nm}(\D_X) = |\disc(\NS(X))|$$ and hence the name. We say that $X$ has \textit{big discriminant} if the natural map $\mu(E) \rightarrow (\Oo_E / \D_X )^{\times}$ is injective, where $\mu(E)$ denotes the roots of unity in $E$. In Remark \ref{remarks discriminant} we show that this is equivalent to the natural map $\Aut(X) \rightarrow \Aut(\NS(X))$ being injective (for K3 surfaces for which the map $\Aut(X) \rightarrow \Aut(\NS(X))$ is not injective, see \cite{vorontsov1983automorphisms, zbMATH00039556, SCHUTT2010206}). We note that having big discriminant is a general condition, in the sense that, for most CM fields $E$, one has $\mu(E) = \{ \pm 1 \}$, and in this case $X$ has big discriminant as long as $\NS(X)$ is not a $2$-elementary lattice. The next theorem is the main result of our paper. 
 \begin{thm} \label{main theorem introduction}
Let $X / \C$ be a K3 surface with complex multiplication by $\Oo_E$, where we consider $E \subset \C$, and assume that $X$ has big discriminant. Then $X$ admits a model $X^{\text{can}}$ over $K := F_{\D_X}(E)$ such that $G_K := \Gal(\overline{K}/K)$ acts trivially on $\NS(\overline{X}^{\text{can}})$. Moreover, $X^{\text{can}}/K$ is uniquely determined by the following property: if $Y$ is a K3 surface over a number field $L$, with CM over $L$, such that $Y_{\C} \cong X$ and $G_L$ acts trivially on $\NS(\overline{Y})$ then $F_{\D_X}(E) \subset L$ and $X^{\text{can}}_{L} \cong Y.$ 
\end{thm}
\begin{rmk} \label{remark main theorem introduction}
Let $\rho(X)$ be the Picard rank of a K3 surface $X$. The fact that $G_K$ acts trivially on $\NS(\overline{X}^{\text{can}})$ means that $\rho(\overline{X}^{\text{can}}) = \rho(X^{\text{can}})$ and this prevents $F_{\D_X}(E)$ from being the smallest field of definition for $X$. On the other hand, the difference between a smaller field of definition of $X$ and $F_{\D_X}(E)$ can be uniformally bounded: there exists an effectively computable constant $C > 0$ such that, for every K3 surface $X$ over a number field $L$, there is a field extension $L'/L$ of degree  $[L':L] \leq C$ such that $G_{L'}$ acts trivially on $\NS(\overline{X})$ (see Huybrechts' book \cite{MR3586372}, p. 393). If $X$ as in Theorem \ref{main theorem introduction} can also be defined over a number field $L$, then the property in the theorem implies that $K \subset E \cdot L',$ and in particular $$\frac{1}{20 C} [K \colon \Q] \leq [L \colon \Q],$$
since $[E \colon \Q] \leq 20$ always. 
\end{rmk}
We classify K3 surfaces with $\rho(X) = 20$ to which Theorem \ref{main theorem introduction} applies.
\begin{prop} \label{most algebraic}
Let $X/\C$ be a K3 surface with complex multiplication by the ring of integers of an imaginary quadratic field $E$. Then $X$ has big discriminant unless 
\begin{itemize}
\item $T(X) \cong$ \( \begin{bmatrix}
2 & 0 \\
0 & 2 \\
\end{bmatrix} \), and therefore $E=\Q(i)$ or 
\item  $T(X) \cong$  \( \begin{bmatrix}
2 & 1 \\
1 & 2 \\ 
\end{bmatrix} \) and $E= \Q(\sqrt{-3})$.
\end{itemize}
\end{prop}
Thus, there are only two such surfaces, and they correspond to the ones studied by Vinberg \cite{MR719348}. When $[E \colon \Q] > 2$, we show that most K3 surfaces satisfy the big discriminant assumption. 
\begin{prop}
Let $E$ be a CM number field, and denote by $K(E)$ the set of isomorphism classes of K3 surfaces over $\C$ with CM by the ring of integers of $E$. Then, up to finitely many elements, every $X \in K(E)$ has big discriminant. 
\end{prop}
\begin{rmk}
From \cite{VALLONI2021107772}, Proposition 7.11, the set $K(E)$ is infinite whenever $[E \colon \Q] \leq 10$. 
\end{rmk}
When $X$ does not have big discriminant, it is not possible to determine a canonical field of definition for $X$, but one can still descend $X$ after fixing a level structure on its transcendental lattice. In Theorem \ref{DescendingI} we give a more general form of our main result.

We proceed to determine some properties of $X^\text{can}$. In Proposition \ref{Proposition Galois representation} we compute the Frobenius elements in the $\ell$-adic Galois representation of $X^{\text{can}}$ that correspond to prime ideals of $K$ coprime to $\ell$ and $\disc(\NS(X))$. Then in \ref{picard section} we study when the inclusion $\Pic(X^{\text{can}}) \subset \Pic(\overline{X}^{\text{can}})$ is an equality. 

\subsubsection{Asymptotics on fields of definition} \label{subsection introduction asymptotics}
In the following, we denote by $F \subset E$ the maximally totally real subfield of $E$ and by $d_X \coloneqq \D_X \cap \Oo_F$. Moreover, for any ideal $I \subset \Oo_E$ we put $$\phi_E(I) \coloneqq |(\Oo_E/I)^\times|$$ and we do similarly for $F$. Finally, let $h_E$ and $h_F$ be the class numbers of $E$ and $F$ respectively. 

\begin{thm} \label{theorem asymptotics introduction}
There are computable universal constants $A_1,A_2 > 0$ such that, for any K3 surface $X/ \C$ with CM by any $\Oo_E$, if $F_X$ is the minimal degree of a field of definition for $X$, then $$A_1 \cdot \frac{\phi_E(\D_X)}{\phi_F(d_X)} \cdot \frac{h_E}{h_F} \leq F_X \leq A_2 \cdot \frac{\phi_E(\D_X)}{\phi_F(d_X)} \cdot \frac{h_E}{h_F}.$$
In particular, since $h_E^- \coloneqq \frac{h_E}{h_F} $ is always an integer, for any $X$ one has
$$A_1 \cdot \frac{\phi_E(\D_X)}{\phi_F(d_X)} \leq F_X.$$
\end{thm}
The universal above means that $A_1$ and $A_2$ work for any $X$ and any $\Oo_E.$  Using \cite[Lemma 4]{zbMATH04077384}, which is a corollary of the Brauer-Siegel theorem, we are then able to prove Orr and Skorobogatov's result in the principal case:
\begin{thm} \label{finiteness}
Let $N >0$ be an integer. There are only finitely many $\C$-isomorphism classes of CM, principal K3 surfaces that can be defined over an extension $K / \Q$ of degree at most $N$. 
\end{thm}
\subsubsection{Classifying CM K3 surfaces with small fields of definition} \label{subsection introduction small fields}
In Section \ref{subsection Principal CM K3 surfaces with small fields of definition} we classify K3 surfaces with CM by $\Oo_E$ that admit a model over $E$ of full Picard rank. The fields of definition of K3 surfaces of maximal Picard rank were completely determined by Sch\"{u}tt in \cite{MR2346573} and \cite{MR2602669}. Since our techniques are different from his, it is interesting to compare the final results. As an instance, Theorem 2 and Lemma 33 of \cite{MR2602669} together with our main result imply that when $E$ is quadratic imaginary one has $F_{\D_X}(E) = H(d)$, where $H(d)$ is the ring class field of the only quadratic order of discriminant $d = \disc(T(X))$. We conclude the paper by giving a direct proof of the equality $F_{\D_X} = H(d)$ under the assumption that $d$ is a fundamental discriminant.

\subsubsection{Notation and conventions}
We adopt the following notation. 
\begin{itemize}
\item We denote by $\widehat{\Z} := \varprojlim_{n} \Z / n \Z$ the profinite completion of $\Z,$ and with $\A_f := \Q \otimes \widehat{\Z}$ the finite ad\`{e}les of $\Q.$
\item Similarly for any number field $E$, we denote by $\widehat{\Oo}_E := \widehat{\Z} \otimes \Oo_E$ the integral finite ad\`{e}les of $E$, and $\A_{E,f} := \A_f \otimes E$ the finite ad\`{e}les of $E$.

\item If $R$ is a ring, $R^\times$ denotes the group of invertible elements. 
    \item For a finitely generated $\Z$-module $N$, we denote $\widehat{N} := N \otimes \widehat{\Z}$ for its profinite completion. For example, $\widehat{T}(X)$ will denote the profinite completion of the transcendental lattice of a K3 surface $X.$
    \item By a lattice we mean a finitely generated, free $\Z-$module $N$ endowed with an integral, symmetric, non-degenerate quadratic form. Its signature is the signature of $N_\R$, and $\mathrm{O}(N) \subset \Gl(N)(\Z)$ is its isometry group. 
    \item For a K3 surface $X/ \C$ with CM, we write $E_X \coloneqq \End_{\Hdg}(T(X)_\Q)$ and $\Oo_X \coloneqq \End_{\Hdg}(T(X))$;
    \item For a K3 surface $X/ \C$ with CM, one can always consider the CM field $E_X$ canonically embedded into $\C$. We showed in \cite{VALLONI2021107772} that $E_X \subset \C$ corresponds to the reflex field of $T(X)$. So we shall refer to $E_X$ as to the reflex field of $X$ when we want to emphatise that $E$ is a subfield of $\C$.
\end{itemize}
Finally, all the number fields in the paper are to be considered inside $\C$. In fact, we shall always start with a K3 surface $X/ \C$, so that $\C$ is part of the data. Moreover, since all the number fields we consider are abelian extensions of $E$, and since $E \subset \C$ naturally by the previous point, we can consider all these abelian extensions as subfields of $\C$ in a natural way. 
\section*{Acknowledgement}
This work was written while the author was a PhD student at Imperial College London. The research was funded by an EPSRC studentship (project reference EP/N509486/1). 

\section{Preliminaries} \label{Preliminaries}
In this section we introduce the main objects of the paper. 
\subsubsection{K3 surfaces and lattices}
Let $X$ be a complex K3 surface. The second Betti cohomology group of $X$ together with the intersection form $$(-,-)_X \colon \mathrm{H}_B^2(X, \Z(1)) \times \mathrm{H}_B^2(X, \Z(1)) \rightarrow \mathrm{H}_B^4(X, \Z(2)) \cong \Z$$ is an even unimodular lattice of rank $22$ and signature $(3,19)$, whose isomorphism class does not depend on the chosen $X$. It is usually denoted by $\Lambda_{K3}$ and called the K3 lattice. Using the fact that $X$ is simply connected, one can show that the natural quotient map $\Pic(X) \twoheadrightarrow \NS(X)$ is an isomorphism, where $\NS(X)$ is the N\'{e}ron-Severi group of $X$. Therefore, the first Chern-class map provides a primitive embedding of lattices $c_1 \colon \Pic(X)  \cong \NS(X) \hookrightarrow \mathrm{H}_B^2(X, \Z(1))$, and the orthogonal complement of $\NS(X)$ in  $\mathrm{H}_B^2(X, \Z(1))$ is the transcendental lattice $T(X) := \NS(X)^{\perp}$, which is an even lattice of signature $(2,22 - \rho)$, where $\rho = \mathrm{rank}(\NS(X))$ is the Picard number of $X.$
As explained by Nikulin in part 3 of Chapter 1 of \cite{MR525944}, in this kind of situation it is natural to introduce the finite quadratic form associated to $\NS(X)$. In order to construct it consider first the dual lattice of $\NS(X)$: $$\NS(X)^{\vee} := \{ x \in \NS(X)_{\Q} \colon (x,v)_X \in \Z \,\, \text{for all} \,\, v \in \NS(X) \} \cong \Hom(\NS(X), \Z)$$ and put $A_N := \NS(X)^{\vee}  / \NS(X)$, under the canonical inclusion $\NS(X) \subset \NS(X)^{\vee}$. Then one define a quadratic form $q_N$ on $A_N$ by the rule 
$$q_N( x + \NS(X)) = (x,x)_X + 2 \Z,$$
which makes sense because $\NS(X)$ is even. With the same procedure one can associate a finite quadratic form $(A_T, q_T)$ to $T(X)$ as well. Nikulin then proved  ( \cite[1.6.1]{MR525944} ) that the embedding $T(X) \oplus \NS(X) \subset \mathrm{H}_B^2(X, \Z(1))$ induces a natural identification 
\begin{equation} \label{identification discriminants}
(A_T, -q_T) \cong (A_N, q_N)
\end{equation}
and that, on the contrary, any embedding of $\NS(X)$ into an even unimodular lattice with orthogonal complement isometric to $T(X)$ is determined by such an isomorphism. 
\begin{defi} \label{defiform}
The finite quadratic form $(A_N, q_N) \cong (A_T, -q_T) $ is called the discriminant form of $X$, and we denote it by $(D_X, q_X)$ (we drop the subscript `$X$' if no confusion can arise). The group of isomorphism of $D$ preserving $q$ is denoted by $O(q)$. We have natural maps 
$ d_N \colon O(\NS(X)) \rightarrow O(q) $ and $ d_T \colon O(T(X)) \rightarrow O(q) $, where the latter is constructed using the identification \eqref{identification discriminants}. 
\end{defi}

\begin{lemma}[Nikulin] \label{lemma nikulin} \label{lemma}
Two isometries $f_N \in O(\NS(X))$ and $f_T \in  O(T(X))$ can be lifted to a (necessarily unique) isometry $f \in O(  \mathrm{H}_B^2(X, \Z(1)) )$ if and only if $d_N(f_N) = d_T(f_T)$. 
\end{lemma}
\begin{rmks}  \label{remark nikulin lemma}
\begin{enumerate}
\item If $f_T$ is a Hodge isometry (i.e., an isometry that respects the Hodge decomposition) and the lifting $f$ exists, then $f$ is a Hodge isometry as well;
\item It follows that one has a pull-back diagram 
\begin{center}

\begin{tikzcd}
O_{\Hdg}(  \mathrm{H}_B^2(X, \Z(1)) \arrow[d] \arrow[r] & O(\NS(X)) \arrow[d,"d_N"]  \\
O_{\Hdg}(T(X)) \arrow[r, "d_T"] &  O(q).
\end{tikzcd}
\end{center}

\item One can formulate Lemma \ref{lemma} in the \et ale context as well. In this case, the role of $\Z$ is played by $\widehat{\Z}$ (see for example the discussion at Section 7 of Chapter 1 of \cite{MR525944}). One considers the $\widehat{\Z}$-lattices $\widehat{\NS}(X) \coloneqq \NS(X) \otimes \widehat{\Z}$ and $ \widehat{T}(X) \coloneqq T(X) \otimes \widehat{\Z}$, both inside $\mathrm{H}^2_{\et}(X , \widehat{\Z}(1))$. All these groups are naturally endowed with a bilinear pairing with values in $\widehat{\Z}$, and one carries out the very same definitions and computations as before. In particular, since $\widehat{T}(X)^\vee / \widehat{T}(X) \cong D_X$ via the comparison isomorphism between \et ale and singular cohomology, every isometry $f \colon \widehat{T}(X) \rightarrow \widehat{T}(X)$ induces a finite isometry $d_T(f) \in O(q)$. 

\end{enumerate}
\end{rmks}

\subsubsection{Main theorem of complex multiplication.} We follow \cite{2005math......8018R} or \cite{VALLONI2021107772} as references, and we adopt the same notation of the introduction. Let $X/ \C$ be a complex K3 surface with complex multiplication and let $E \coloneqq E_X \subset \C$ be its reflex field. Then, the algebraic group $\Res_{E/\Q} \mathbb{G}_{m,E}$ acts naturally on $T(X)_\Q$, and Zarhin showed in \cite[2.3.1.]{MR697317} that the Mumford-Tate group of $T(X)$ is identified with the torus $U_E \subset \Res_{E/\Q} \mathbb{G}_{m,E}
¸$ defined by the equation $x \overline{x} =1$. Suppose that $X$ can be defined over a number field $L \subset \C$ and that $X$ has complex multiplication over $L$ too, which translates into $E \subset L$.
Attached to $X$ there a Galois representation $\rho \colon G_L \rightarrow \Aut(T(X)_{\Q})(\A_f)$, with image in $U_E(\A_f)$, and the main theorem of complex multiplication describes $\rho$ in terms of class field theory. We recall that class field theory provides us with a commutative diagram
\begin{center}

\begin{tikzcd}
\A_{L}^{\times} \arrow[d, "\mathrm{Nm}_{L/E}"] \arrow[r, "\art_L"] & G_L^{ab} \arrow[d, "\res_{K/E}"]  \\
\A_{E}^{\times} \arrow[r, "\art_E"] & G_E^{ab},
\end{tikzcd}
\end{center}
where $\art_E$ is the Artin map, a surjective, continuous morphism with $E^\times$ is its kernel (and similarly for $L$). 
Note that, since $E$ is a CM field, both the Artin maps factorise through the finite id\`{e}les. We have 
\begin{thm} \label{MTCM}
Let $\tau \in G_L$, $t \in \A_{L,f}^{\times}$ such that $\art_L(t) = \tau_{| L^{\textit{ab}}}$ and put $s := \mathrm{Nm}_{L/E}(t) \in \A_{E,f}^{\times}$. There exists a unique $u \in U_E(\Q)$ such that $$\rho(\tau) = u \frac{s}{\overline{s}} \in U_E(\A_f),$$
where $s \mapsto \overline{s}$ is the complex conjugation. 
\end{thm}
Rizov proved the theorem above by establishing it first for K3 surfaces of maximal Picard rank and and then concluding via a density argument. Note that a shorter proof is also provided by Madapusi-Pera \cite[Corollary 4.4]{MR3370622} using the theory of absolute Hodge cycles.  
\begin{rmk}
The map $\rho(\tau)$ is a $\widehat{\Z}$-linear isometry of $\widehat{T}(\overline{X})$, because the intersection form has value in $\mathrm{H}^4_\et(\overbar{X}, \widehat{\Z}(2)) \cong \widehat{\Z}$, which has trivial Galois action. It follows that also multiplication by $u \frac{s}{\overline{s}}$ must be an isometry of $\widehat{T}(\overline{X})$. This implies that $u \frac{s}{\overline{s}} \in \hat{\Oo}_E^\times$, and therefore thanks to the second remark after Lemma \ref{lemma}, it makes sense to consider the induced map $d_T \big(u \frac{s}{\overline{s}} \big) \in O(q_X)$. A direct consequence of the theorem above is that 
\begin{equation} \label{agree on discriminant}
d_T \bigg(u \frac{s}{\overline{s}} \bigg) = d_N(\tau^*|_{\NS}),
\end{equation}
where $\tau^*|_{\NS} \colon \NS(\overline{X}) \rightarrow \NS(\overline{X})$ denotes the Galois action on the N\'{e}ron-Severi group. 
\end{rmk} 
\subsubsection{K3 class fields} \label{K3 class filds subsection}
Let $F \subset E$ be the maximal totally real subextension of $E$ and let $I \subset \Oo_E$ be an ideal. The fields $F_{I}(E)$ are finite Abelian extensions of $E$, therefore we can describe them using class field theory. In fact, let $K$ be any finite Abelian extension of $E.$ Since $E$ is a CM field, the map $\art_E$ induces an isomorphism $$ \Gal(K/E) \cong \frac{\A^\times_{E,f}}{E^\times \cdot \mathrm{Nm}_{K/E}(\A^\times_{K,f})},$$ and giving the finite-index subgroup $E^\times \cdot \mathrm{Nm}_{K/E}(\A^\times_{K,f})$ of $\A^\times_{E,f}$ is the same as giving $K$. For what concerns $F_I(E)$, we have
\begin{enumerate}
\item The norm group of $F_{I}(E)$ corresponds to $$S_I := \{ s \in \A^{\times}_{E,f} \colon \exists u \in U(\Q) \colon u \frac{s}{\overline{s}} \Oo_E = \Oo_E \,\, \text{and} \,\, u \frac{s}{\overline{s}} \equiv 1 \mod I \}.$$ Since $ S_{I} = S_{\overline{I}} = S_{\lcm(I , \overbar{I})},$ we can choose $I$ such that $I = \overline{I}$ without loss of generality. 
\item Let $K_I(E)$ denote the ray class field modulo the ideal $I$, and let $\Cl_I(E)$ the ray class group modulo $I$, so that $\Gal(K_I(E) / E) \cong \Cl_I(E)$. 
Since $I = \overline{I}$ by assumption, complex conjugation acts on $\Cl_I(E)$. Denote by $\Cl_I'(E)$ the co-invariant of $\Cl_I(E)$, that is, $\Cl_I'(E) \coloneqq \Cl_I(E)/ \Cl_I(E)^G$ where $G$ is the group generated by the complex conjugation, and by $K_I'(E)$ the unique field subextension of $K_I(E)$ such that $\Gal(K_I'(E) / E) = \Cl_I'(E)$. If $E$ is quadratic imaginary, then \cite[Remark 9.2.]{VALLONI2021107772} says that
\begin{equation} \label{K3classgroup}
F_{I}(E) = K_I'(E) 
\end{equation}
\item In general, let $$E^{I,1} = \{ e \in E^\times \colon  e-1 \equiv 0 \mod I \}$$ and let $\Oo_E^{I}:= \Oo_E^\times \cap E^{I,1}$. We have a diagram 
\begin{center}
\begin{tikzcd}[every arrow/.append style=dash] 
& K_I(E) 
 \arrow{dd} & \\
 F_{I}(E) 
  \arrow{dr} & \\
& K_I'(E)
 \arrow{d}\\
& E,
\end{tikzcd}
\end{center}
with 
\begin{equation} \label{K3generalgroup}
\Gal(F_{I}(E) / K_I'(E)) \cong \frac{\Oo_F^\times \cap \mathrm{Nm}_{E/F}(E^{I,1})}{\mathrm{Nm}_{E/F}(\Oo_E^{I})}.
\end{equation}
\item Let $X/ \C$ have complex multiplication by the ring of integers of $E \subset \C$, and define $$T(X)[I] := \{v \in T(X) \otimes \Q / \Z \colon ix = 0 \,\, \forall \,\, i \in I \}.$$
Note that $T(X)[I]$ can also be defined via \et ale cohomology, so in particular it is functorial with respect to any scheme isomorphism. Consider the group 
$$ \mathcal{S}_I := \{ \tau \in \Aut(\C / E) \colon \exists f \colon  T(X^\tau) \xrightarrow{\sim} T(X)\,\, \text{with} \,\,f \circ \tau^* |_{T(X)[I]} = \id\},$$
where $f$ is an integral Hodge isometry, $X^\tau$ is the base-change of $X$ along $\C \xrightarrow{\tau} \C$ and $\tau^* \colon \widehat{T}(X) \otimes \Q/ \Z \rightarrow \widehat{T}(X^\tau)  \otimes \Q/ \Z $ and $f_* \colon \widehat{T}(X^\tau) \otimes \Q/ \Z \rightarrow \widehat{T}(X)  \otimes \Q/ \Z   $ are the natural induced maps. We showed in \cite[Theorem 11.2]{VALLONI2021107772} (see also Remark 4.1 in the same paper) that $F_{I}(E)$ is the fixed field of $ \mathcal{S}_I$. Differently said, $F_{I}(E)$ is the field of moduli over $E$ of the pair $(T(X), T(X)[I]).$
\end{enumerate}

\section{Discriminant ideal and K3 surfaces with big discriminant}
In this section we introduce the discriminant ideal of a principal K3 surface $X / \C$ with complex multiplication. Let $X/ \C$ be a K3 surface with complex multiplication by the maximal order of $E_X = \End_{\Hdg}(T(X)_\Q)$. Recall that $\Oo_X = \End_{\Hdg}(T(X))$. Consider \begin{equation} \label{Inverse different}
\D_X^{-1} \coloneqq \{ f \in E_X \colon (f(v), w)_X \in \Z \,\, \text{for all} \,\, v,w \in T(X) \},
\end{equation} 
where $(-,-)_X$ is the intersection pairing on $T(X).$ It is readily checked that $\Oo_X \subset \D^{-1}_X$ and that $\D^{-1}_X$ is a fractional ideal of $E_X.$ Note moreover that $\overline{\D_X^{-1} } = \D_X^{-1} $ since $(f(v), w)_X = (v, \overline{f}(w))_X.$

\begin{defi} \label{defi discriminant ideal}
The discriminant ideal $\D_X \subset \Oo_X$ is by definition the inverse of $\D_X^{-1}$:

$$ \D_X := \{ f \in E_X \colon f \cdot \D_X^{-1} \subset \Oo_X \}.$$

\end{defi} 
If $T(X)^\vee \subset T(X)_\Q$ is the dual of $T(X)$, then the definition above is equivalent to 
\begin{equation} \label{equation discriminant}
    \D_X = \{ f \in E_X \colon f (T(X)^\vee) \subset T(X) \}.
\end{equation}
Another and more direct description of $\D_X$ can be given using the notion of type of a CM K3 surface, introduced in Section 7 of \cite{VALLONI2021107772}. A type is a linear data on the CM field $E_X$ that completely classifies $T(X)$ as an integral, polarized Hodge structure. The notion of type is analogous to the one used in the theory of CM Abelian varieties: let $(E, \sigma)$ be a couple consisting of a CM field $E$ with an embedding $\sigma \colon E \hookrightarrow \C$, and let $(X, \iota)$ be a principal CM K3 surface with an isomorphism $\iota \colon E \xrightarrow{\sim} E_X = \End_{\Hdg}(T(X)_\Q)$. We denote by $\sigma_X \colon E_X \rightarrow \C$ the natural embedding and we consider $T(X)$ as an $\Oo_E$-module via the map $\iota$. 
\begin{defi} \label{TYPE}
Let $\alpha \in F^\times$ and let $I \subset E$ be a fractional ideal of $E$. We say that $(T(X), \iota)$ is of type $(I, \alpha, \sigma)$ if there exists an isomorphism of $\Oo_E-$modules $$\phi \colon T(X) \xrightarrow{\sim} I $$ such that, if $(-,-)_X$ denotes the intersection pairing on $T(X)$, one has: 
\begin{enumerate}
\item $(v,w)_{X} = \tr_{E/ \Q} \Big( \alpha \phi(v) \overline{\phi(w)} \Big)$ for every $v,w \in T(X)$;
\item $\sigma_X \circ \iota = \sigma. $
\end{enumerate}

\end{defi}
Note that $(I, \alpha, \sigma)$ is an integral, polarized Hodge structure: the polarization $I \times I \rightarrow \Z$ is given by $(x,y) \mapsto \tr_{E/\Q}(\alpha x \overline{y})$, and the Hodge decomposition is determined by 
$$I^{2,0} := \{ v \in I_\C \colon e \cdot v = \sigma(e)v \,\, \text{for all} \,\, e \in E \}.$$
One can readily verify that every principal CM K3 surfaces has a type and that, if $\iota$ is fixed, two different types $(I, \alpha, \sigma)$ and $(J, \beta, \sigma)$ represents the same transcendental lattice if and only if there exists $e \in E^\times$ such that $eJ=I$ and $\beta = e \overline{e} \alpha.$

\begin{prop}  \label{discriminantP}
Let $(X, \iota)$ be of type $(I, \alpha, \sigma)$. Then 
\begin{enumerate}
\item Under the map $\iota$ the ideal $\D_X$ corresponds to $ (\alpha) I \overline{I} \D_E;$
\item The type map $\phi \colon T(X) \rightarrow I$ induces an isomorphism between the $\mathcal{O}_E$-modules $D_X$ and $\mathcal{O}_E / \D_X.$ In particular, $\mathrm{Nm}(\D_X) = \disc(T(X)).$
\end{enumerate}
\end{prop}

\begin{proof}
\begin{enumerate}
\item If we consider $E$ as a rational quadratic space with the quadratic form given by $\tr(\alpha x \overline{y})$, then the map $\phi_\Q \colon I_\Q \cong E \rightarrow T(X)_\Q$ is an isometry that restricts to an isomorphism between $I$ and $T(X).$ Consequently, $\phi_\Q$ also induces an isomorphism between $I^\vee \subset E$ and $T(X)^\vee \subset T(X)_\Q.$ It is straightforward to check that $I^\vee$ is the fractional ideal given by $(\alpha^{-1}) \overline{I}^{-1} \D_E^{-1},$ so we use \eqref{equation discriminant} to write $$\iota^{-1}(\D_X) = \{ e \in E \colon eI^\vee \subset I \} = (\alpha) I \overline{I} \D_E.$$

\item In fact, we have the following isomorphisms $$T(X)^\vee/ T(X) \cong I^\vee / I \cong \Oo_E / (\alpha) I \overline{I} \D_E.$$
\end{enumerate}
\end{proof}
\begin{defi}[Big discriminant]
The group of integral Hodge isometries of $T(X)$ is denoted by $\mu(X)$ and corresponds to the roots of unity in $E_X$: 
$$ \mu(X) = \{ e \in  \Oo_E \colon e \overline{e} = 1 \}.$$
The kernel of the canonical map $d_T \colon \mu(X) \rightarrow O(q_X)$ is denoted by $K_X$ and we say that $X$ has \textit{big discriminant} whether $K_X = 1$.
\end{defi}

\begin{rmks} \label{remarks discriminant}
\begin{itemize}
\item Thanks to the second point in Proposition \ref{discriminantP}, having big discriminant is equivalent to the injectivity of the natural map $\mu(E) \rightarrow (\Oo_E / \D_X )^{\times}$.
\item There is always a natural injection $K_X \hookrightarrow \Aut(X).$ Indeed, for any $\mu \in K_X$, the map $(\mu, \id) \colon T(X) \oplus \NS(X) \rightarrow T(X) \oplus \NS(X)$ can be extended to an integral Hodge isometry $\mathrm{H}^2_B(X, \Z(1)) \rightarrow \mathrm{H}^2_B(X, \Z(1)),$ which in turn is induced by a unique automorphism of $X$, thanks to Torelli Theorem. 
\item It follows that $X$ has big discriminant if and only if the natural map $\Aut(X) \rightarrow \mathrm{O}(\NS(X))$ is injective. 
\end{itemize}
\end{rmks}

\begin{prop} \label{keydiscriminant} 
Let $E$ be a CM number field and let $X / \C$ be a K3 surface with CM by $\Oo_E$. Then $\D_X \subset (2)^{-1}\D_{E/F}.$
If moreover $E$ is quadratic imaginary, then $\D_X \subset \D_{E}.$
\end{prop}

\begin{rmk}
In particular, if $E$ is quadratic imaginary and the map $\mu(E) \rightarrow ( \Oo_E / \D_{E})^{\times}$ is injective, then every K3 surface with CM by $\Oo_E$ has automatically big discriminant. 
\end{rmk}

\begin{proof}
For any fractional ideal $I$ of $E$, let $\mathrm{Nm}_{E/F}(I) \subset F$ be its norm, i.e. the fractional ideal of $F$ generated by the elements $x \overline{x}$ for $x \in I,$ so that we have $\mathrm{Nm}_{E/F}(I) \Oo_E = I \overline{I}.$ Let $(I , \alpha)$ be the type of $X$. Every element of $\mathrm{Nm}_{E/F}(I)$ can be written as a finite sum of elements of the form $fx \overline{x}$, with $f \in \Oo_F$, and we compute $$\text{tr}_{F/\Q}( \alpha f x \overline{x}) = 2^{-1} \text{tr}_{E/\Q}( \alpha f x \overline{x}) \in 2^{-1} \Z,$$
since the quadratic form $(I, \alpha)$ is integral. Therefore, by the property of the discriminant ideal, we must have that $$(\alpha) \mathrm{Nm}_{E/F}(I) \subset (2)^{-1} \D_{F/\Q}^{-1}.$$
If we 'base-change' the inclusion above to $\Oo_E$ we obtain that $$(\alpha) I \overline{I} \subset (2)^{-1} \D_{F/\Q}^{-1} \Oo_E, $$ and the first part of the proposition follows by multiplying both sides by $\D_{E}$. To prove the second statement, we note that the quadratic form $(I, \alpha)$ is also even, so that $\text{tr}_{E/\Q}( \alpha x \overline{x}) \in 2 \Z$. But since $E$ is quadratic imaginary by assumptions, we deduce that $\alpha x \overline{x} \in \Z$ for every $x \in I$. Consider again the fractional ideal $\mathrm{Nm}_{E/\Q}(I)$; every $y \in \mathrm{Nm}_{E/\Q}(I)$ can be written as $y = n_1 x_1 \overline{x}_1 + \cdots + n_k x_k \overline{x}_k$ with $n_i \in \Z$ and $x_i \in I$, so thanks to the computation above we conclude that $(\alpha) \mathrm{Nm}_{E/F}(I) \subset \Z$.
Base-changing the above equation to $\Oo_E$, we obtain $$(\alpha) I \overline{I} \subset \Oo_E,$$
and the claim follows as before.

\end{proof}

\section{Descending K3 surfaces} \label{descending section}
In this section, we prove the main result of the paper. Let $X/ \C$ be a K3 surface with CM by $\Oo_E$. We fix an ideal $I \subset \D_X$ such that 
\begin{itemize}
\item $ \overline{I} = I$;
\item The map $\mu(E) \rightarrow (\Oo_E / I)^{\times}$ is injective.
\end{itemize}

\begin{rmk} \label{remark (3)}
Since any finite subgroup $G \subset \Gl(n,\Z)$ injects into $\Gl(n, \Z / 3\Z)$ via the natural reduction map, one can always choose $I = (3)\D_X$ for any K3 surface $X$.
\end{rmk}
Our main result is the following. 
\begin{thm} \label{DescendingI}
Let $X / \C$ be a principal K3 surface with CM and reflex field $E \subset \C$. Then $X$ admits a model $X_I$ over $K := F_{I}(E)$, such that $G_K$ acts trivially on $\NS(\overline{X}_I)$ and $T(\overline{X}_I)[I]$.  
Moreover, $X_I/K$ is uniquely determined by the following  property: if $Y$ is a K3 surface over a number field $L$, with CM over $L$, such that $Y_{\C} \cong X$, and $G_L$ acts trivially on $T(\overline{Y})[I]$ and $\NS(\overline{X})$, then $F_{I}(E) \subset L$ and $X_{I,L} \cong Y$. 
\end{thm}

\begin{rmk}
Since $G_K$ acts trivially on $T(\overline{Y})[I]$, it acts trivially also on $T(\overline{X}_I)[\D_X] \cong D_{\overline{X}_I}$, since $I \subset \D_X$. 
\end{rmk}
In case $X$ has big discriminant, we can choose $I = \D_X$ in the theorem above. This, together with the above remark, leads to the following corollary.

\begin{cor}[Canonical models] \label{bigggg}
Let $X / \C$ be a K3 surface with complex multiplication by the ring of integers of a CM field and denote by $E \subset \C$ its reflex field. Assume that $X$ has big discriminant. Then $X$ admits a model $X^{\text{can}}$ over $K=F_{\D_X}(E)$, the K3 class field of $E$ modulo the discriminant ideal $\D_X$. Moreover, $X^{\text{can}}/K$ is uniquely determined by the following property: if $Y$ is a K3 surface over a number field $L$, with CM over $L$, such that $Y_{\C} \cong X$ and $\rho(Y/L) = \rho(X/ \C)$, then $F_{\D_X}(E) \subset L$ and $X^{\text{can}}_{L} \cong Y$. 
\end{cor}

In order to prove \ref{DescendingI} we shall construct a Galois descent data for $X$ over $K$ using the global Torelli Theorem and the main theorem of complex multiplication. Before this, we need to study the field of definition of isomorphisms.  
\begin{prop} \label{descendingiso}
Let $X,Y / L$ be two principal K3 surfaces with complex multiplication over a number field $L$, with $L \subset \C$, and suppose that $\overline{X}$ and $\overline{Y}$ are isomorphic. 
Then an isomorphism $f \colon \overline{X} \rightarrow \overline{Y}$ is defined over $L$ if and only if the induced maps 
$$ f^* \colon \NS(\overline{Y}) \rightarrow \NS(\overline{X}) $$
and
$$ f^* \colon T(\overline{Y})[I] \rightarrow T(\overline{X})[I] $$ 
are $G_L$-invariant.
\end{prop}

\begin{proof}
The only if part of the statement is trivial, so that what we have to prove is that if the natural maps $ \NS(\overline{Y}) \rightarrow \NS(\overline{X})$ and $T(\overline{Y})[I] \rightarrow T(\overline{X})[I] $ are Galois invariant, then $f$ is defined over $L$. 
Recall that $f$ is defined over $L$ if and only if the induced map $f^* \colon \mathrm{H}^2_{\et}(\overline{Y}, \widehat{\Z})(1) \rightarrow \mathrm{H}^2_{\et}(\overline{X}, \widehat{\Z})(1)$ is $G_L$-invariant, since the natural morphism of $G_L-$modules
$$\Aut(\overline{X}) \rightarrow \Aut(\mathrm{H}_{\et}^2(\overline{X}, \widehat{\Z}))$$ is injective (see Chapter 15, Remark 2.2. of \cite{MR3586372}). In order to check that $f^*$ is Galois invariant, we break it into two parts: $f^*_{T} \colon \widehat{T}(\overline{Y}) \rightarrow \widehat{T}(\overline{X})$ and $f^*_{N} \colon \NS(\overline{Y}) \rightarrow \NS(\overline{X})$. For any $\tau \in G_{L}$, it thus suffices to prove the commutativity of the following two squares:

\begin{center}
\begin{tikzcd}
\widehat{T}(\overline{Y}) \arrow[d, "\tau_Y^*|_T"] \arrow[r, "f^*_T "] & \widehat{T}(\overline{X})  \arrow[d, "\tau_X^*|_T"]  \\
\widehat{T}(\overline{Y}) \arrow[r, "f^*_T"] & \widehat{T}(\overline{X}).
\end{tikzcd}
\,\,\,\,\,\,\,\,\,\,\,\, and \,\,\,\,\,\,\,\,\,\,\,\,
\begin{tikzcd}
\NS(\overline{Y}) \arrow[d, "\tau_Y^*|_{\NS}"] \arrow[r, "f^*_{N}"] & \NS(\overline{X}) \arrow[d, "\tau_X^*|_{\NS}"]  \\
\NS(\overline{Y}) \arrow[r, "f^*_{N}"] & \NS(\overline{X}) .
\end{tikzcd} 
\end{center}
The latter commutes by assumption, so that we can focus on the first. Since $X$ and $Y$ are geometrically isomorphic, the fields $E$, $E(X)$ and $E(Y)$ are naturally identified (i.e., the same field acts on $T(X_\C)_\Q$ and $T(Y_\C)_\Q )$. Let $s \in \A^{\times}_E$ be as in Theorem \ref{MTCM}, and let $e,c \in U(\Q)$ be the unique elements such that $\tau_X^* = e \frac{s}{\overline{s}}$ and $\tau_Y^* = c \frac{s}{\overline{s}}$. The map $f_T^*$ is $\A_E$-linear, so the commutativity condition $$(f^*_T) ^{-1} \circ e \frac{s}{\overline{s}} \circ (f^*_T) =  c \frac{s}{\overline{s}}$$ simply amounts to $e = c$. But both $ e \frac{s}{\overline{s}}$ and  $c \frac{s}{\overline{s}}$ respect the $\widehat{\Z}$-lattice $\widehat{T}(\overline{Y}) $, so $e / c$ must do the same. This, together with the fact that $e \overline{e} = c \overline{c} =1$, implies that $e / c$ is a root of unity, i.e. an integral Hodge isometry of $T(\overline{Y})$. By assumptions, the induced map $T(\overline{Y})[I] \rightarrow T(\overline{X})[I] $ is Galois equivariant, therefore $e/c \equiv 1 \mod I$. Since we chose $I$ such that $\mu(E) \rightarrow (\Oo_E / I)^{\times}$ is injective, we conclude that $e=c$. 
\end{proof}

\begin{rmk}
In case $\overline{X}$ has big discriminant, the proposition says that an isomorphism $f \colon \overline{X} \rightarrow \overline{Y}$ is defined over $L$ if and only if the induced map $ f^* \colon \NS(\overline{Y}) \rightarrow \NS(\overline{X}) $ is $G_L$- equivariant. 
\end{rmk}

\begin{cor} \label{corollary}
Let $X,Y / L$ be two principal K3 surfaces with CM over a number field $L$, and suppose that $\overline{X}$ and $\overline{Y}$ are isomorphic. Suppose, moreover, that $X$ the $G_L$ modules $\NS(\overline{X})$, $\NS(\overline{Y})$, $T(\overline{Y})[I]$ and $ T(\overline{X})[I] $ are trivial. Then every isomorphism $f \colon \overline{X} \rightarrow \overline{Y}$ is already defined over $L$. 
\end{cor}
We recall that a Galois-descent data for $X/ \overline{\Q}$ over a number field $K \subset \overline{\Q}$ consists of an isomorphism $f_{\tau} \colon X^\tau \xrightarrow{\sim} X$ for any $\tau \in G_K$, such that  $f_{\sigma \tau}$ is the composition $$f_{\sigma \tau} \colon X^{\sigma \tau} = (X^\tau)^\sigma \xrightarrow{f_{\tau}^\sigma} X^\sigma \xrightarrow{f_\sigma} X,$$ for any $\sigma, \tau \in G_K$ (e.g., see \cite[4.4.4]{poonen2017rational}). Any model of $X$ over $K$ gives rise to a Galois descent data: let $Y/K$ be a variety with an isomorphism $f \colon Y_{\overline{\Q}} \xrightarrow{\sim} X.$ For any $\tau \in G_K$ the conjugate $Y_{\overline{\Q}}^\tau$ is equal to  $Y_{\overline{\Q}}$ itself, because $Y$ is the base-change of a $K$-scheme. So we obtain a second isomorphism $f^\tau \colon Y_{\overline{\Q}}^\tau = Y_{\overline{\Q}} \xrightarrow{\sim} X^\tau $ and we define $f_{\tau} \coloneqq f^\sigma \circ f^{-1}.$ In our situation, it follows from Corollary 4.4.6. and Remark 4.4.8. of \cite{poonen2017rational} that also the converse is true: to any Galois-descent data for $X$ over $K$ there exists associated a scheme $Y/K$ with an isomorphism $Y_{\overline{\Q}} \cong X$. We can now proceed to prove Theorem \ref{DescendingI}.

\begin{proof}[Proof of Theorem \ref{DescendingI}]
Let $\tau \in \Aut(\C / E)$ and $s \in \A_{E,f}^{\times}$ be such that $\art_{E}(s) = \tau_{|E^{\textit{ab}}}$. From Theorem \ref{MTCM}, there is a unique rational Hodge isometry $\eta(s) \colon T(X)_{\Q} \rightarrow T(X^{\tau})_{\Q}$ such that the following diagram commutes: 

\begin{center}
\begin{tikzcd}[row sep= large, column sep =large]
\widehat{T}(X)_{\Q} \arrow[r, "\eta(s) \otimes \A_{f}"] & \widehat{T}(X^{\tau})_{\Q} \\
\widehat{T}(X)_{\Q} \arrow[u, "\frac{s}{\overbar{s}}"]\arrow[ur, "\tau^*_{|T}"].
\end{tikzcd} 
\end{center}
Our first step to construct the various $f_\tau$'s amounts to determine those $\tau \in \Aut(\C / E)$ for which there exists $s \in \A_{E,f}^{\times}$ such that the rational map $\eta(s)$ is actually \textit{integral} and extends to a global Hodge isometry between $H^2_B(X, \Z)(1)$ and $H^2_B(X^\tau, \Z)(1)$. Let $e \in E^{\times}$. Since $\art_{E}(s) = \art_{E}(es)$, if we operate the substitution $s \mapsto es$ we must obtain that

\begin{equation} \label{equazione eta}
    \eta(s) = \frac{e}{\overline{e}}\eta(es)
\end{equation}

Suppose now that we can find $e \in E^{\times}$ that satisfies the following properties: 
\begin{enumerate}
\item $\frac{e s}{\overline{e} \overline{s}} \in \widehat{\Oo}^\times_E$
\item $\frac{e s}{\overline{e} \overline{s}} \equiv 1 \mod I$
\end{enumerate}
and for any $s$ denote by $E(s) \subset E^\times$ the set of elements $e \in E^\times$ satisfying (1) and (2) above.
\begin{claim}
If $s$ is such that $E(s)$ is not empty, then the map
\begin{align} \label{per i posteri}
 E(s) &\rightarrow U(\Q)  \\
 e &\mapsto \frac{e}{\overline{e}} \nonumber
\end{align} 
is constant. 
\end{claim}  
\begin{proof}[Proof of Claim]

Indeed, let $e, f \in E(s)$ and put $x:=\frac{e \overline{f}}{\overline{e} f}$. By the first point above, we have that $x \Oo_E = \Oo_E,$ i.e. $x \in \Oo_E^{\times}$. Since $\overline{x} = x^{-1}$, we also have that $x$ is a root of unity. By the second point above, we see that $x \equiv 1 \mod I$. Hence $x=1$, since we have chosen $I$ such that $\mu(E) \rightarrow (\Oo_E /I)^\times$ is injective. 
\end{proof}
It follows that to any element $s \in \A^{\times}_{E,f}$ such that $E(s)$ is not empty we can associate a unique Hodge isometry $$\eta'(s) \colon T(X)_{\Q} \rightarrow T(X^{\tau})_{\Q} $$
and a unique element $\rho(s) \in U(\A_f)$ by putting $\eta'(s) := \eta(es)$ and $\rho(s):= \frac{es}{\overline{es}}$, where $e \in E(s)$ is a random element (by the claim above and \eqref{equazione eta}, both $\eta'(s)$ and $\rho(s)$ do not depend on the choice).

\begin{claim} \label{claim proof descendingI}
The map $\eta'(s)$ is integral, and the isometry $$(\eta'(s), \tau^*) \colon T(X) \oplus \NS(X) \rightarrow T(X^\tau) \oplus \NS(X^\tau)$$ extends to a Hodge isometry $f(s) \colon H^2_{B}(X, \Z(1)) \xrightarrow{\sim} H^2_{B}(X^\tau, \Z(1)),$ where $\tau^* \colon \NS(X) \rightarrow  \NS(X^\tau)$ is the Galois pullback on divisors. 
\end{claim}

\begin{proof}[Proof of Claim]
By construction, the following diagram commutes

\begin{equation} \label{diagramma} 
\begin{tikzcd} 
\widehat{T}(X)_{\Q} \arrow[r, "\eta'(s) \otimes \A_{f}"] & \widehat{T}(X^{\tau})_{\Q} \\
\widehat{T}(X)_{\Q} \arrow[u, "\rho(s)"]\arrow[ur, "\tau^*_{| T}"].
\end{tikzcd} 
\end{equation}
Since for our choice of $s$ we have that $\rho(s) \in \widehat{\Oo}^\times_E$, we notice that both maps $\rho(s)$ and $\tau^*$ are integral, in the sense that $\rho(s)$ restricts to an isometry of $\widehat{T}(X)$, and $\tau^*$ restricts to an isometry $\widehat{T}(X) \xrightarrow{\sim} \widehat{T}(X^\tau)$. This implies that $$\eta'(s) \in \Hom(T(X)_\Q, T(X^\tau)_\Q) \cap \Hom(\widehat{T}(X), \widehat{T}(X^\tau)) \subset \Hom(\widehat{T}(X)_\Q, \widehat{T}(X^\tau)_\Q),$$ i.e., that $\eta(s)'$ determines an integral Hodge isometry $T(X) \xrightarrow{\sim} T(X^\tau)$. To prove the last statement of the claim, we make use of \ref{lemma nikulin} and the Remark \ref{remark nikulin lemma}. We note that the map $\tau^* \colon H^2_{\et}(X, \widehat{\Z}(1)) \rightarrow  H^2_{\et}(X^\tau, \widehat{\Z}(1))$ is an integral (ad\'{e}lic) isometry. Let $D_X$ and $D_{X^\tau}$ be the discriminant groups of $X$ and $X^\tau$ respectively. To avoid any possible confusion, we denote by $\tau^*_{\NS} \colon \NS(X) \rightarrow \NS(X^\tau)$ the induced map on the N\'{e}ron-Severi groups and by $\tau^*_{T} \colon \widehat{T}(X) \rightarrow \widehat{T}(X^\tau)$ the induced map on the profinite transcendental lattices. Then both $\tau^*_{\NS}$ and $\tau^*_{T}$ induce maps $$d_N(\tau^*_{\NS}), d_T(\tau^*_{T})  \colon D_X \rightarrow D_{X^{\tau}},$$ and since $\tau^*_{\NS} \oplus \tau^*_{T}$ extends to a global isometry $H^2_{\et}(X, \widehat{\Z}(1)) \rightarrow  H^2_{\et}(X^\tau, \widehat{\Z}(1))$ we conclude that $d_N(\tau^*_{\NS}) = d_T(\tau^*_{T}).$ Using \eqref{diagramma} we factorize $\tau^*_{T}$ as $$\tau^*_{T} \colon \widehat{T}(X) \xrightarrow{\rho(s)} \widehat{T}(X) \xrightarrow{\eta'(s)} \widehat{T}(X^\tau).$$ Since $\rho(s) \equiv 1 \mod I$ and $I \subset \D_X$ also $\rho(s) \equiv 1 \mod \D_X$ holds. This implies that the map induced by $\rho(s)$ on the discriminant group $D_X$ is the identity, so that $d_T( \tau^*_{T}) = d_T(\eta'(s) )$. In particular $$d_T(\eta'(s) ) = d_N(\tau^*_{\NS}),$$ and the claim follows from \ref{lemma nikulin}. 
\end{proof}
Note that since $\tau^*_{\NS} \colon \NS(X) \rightarrow \NS(X^\tau)$ maps ample classes to ample classes, by the global Torelli theorem there is a unique isomorphism $f_s \colon X^\tau \rightarrow X$ the induces $f(s)$ in cohomology. We are thus very close to something that looks like a descent data. The elements $s \in \A^\times_{E,f}$ such that $E(s)$ is not empty are readily determined:
$$\{ s \in \A_{E,f}^{\times} \colon E(s) \neq \emptyset \} = \lbrace s \in \A^{\times}_{E,f} \colon \exists e \in E^\times \colon \frac{es}{\overline{es}} \mathcal{O}_{E} = \mathcal{O}_{E}, \,\,   \frac{es}{ \overline{es}} \equiv 1 \mod I \rbrace$$
and thanks to Hilbert's Theorem 90 we can write this group as 
$$\lbrace s \in \A^{\times}_{E,f} \colon \exists u \in U(\Q) \colon u\frac{s}{\overline{s}} \mathcal{O}_{E} = \mathcal{O}_{E}, \,\,  u  \frac{s}{ \overline{s}} \equiv 1 \mod I \rbrace,$$
which is exactly the norm group $S_I$ associated to the Abelian field extension $F_{I}(E) / E$ (see \ref{K3 class filds subsection}). Denote this extension by $K$. 
To show that $f_s$ actually depends only on $\tau \in G_K$, consider the commutative diagram 
\begin{center} 
\begin{equation} \label{square proof descending}
\begin{tikzcd} 
\A_{K,f}^{\times} \arrow[d, twoheadrightarrow, "\mathrm{Nm}_{K/E}"] \arrow[r, twoheadrightarrow, "\art_K"] & G_K^{\textit{ab}} \arrow[d, "\res_{K/E}"]  \\
S_I \arrow[r, "\art_E | S_I"] & G_E^{\textit{ab}}.
\end{tikzcd}
\end{equation}
\end{center}
The map $\rho \colon S_I \rightarrow U(\A_f)$ constructed before is continuous and has the property that $\rho(E^\times) = 1 $. Therefore, it factorises through the profinite completion of $S_I / E^\times$ which is canonically isomorphic to $\art_E(S_I) = \res(G_K^{ab})$. This means that the next dotted arrow can be uniquely filled in a way that the following diagram commutes:

\begin{center}
\begin{tikzcd}
\A_{K,f}^{\times} \arrow[d, twoheadrightarrow, "\mathrm{Nm}"] \arrow[r, twoheadrightarrow, "\art_K"] & G_K^{\textit{ab}} \arrow[d, "\res"] \arrow[ddr, bend left, dotted]  \\
S_I \arrow[r, "{\art_{E}}_{| S_I}"] \arrow[drr, bend right, "\rho"] & G_E^{\textit{ab}} \\
& & U(\A_f)
\end{tikzcd}
\end{center}

We still denote by $\rho \colon G_K^{ab} \rightarrow U(\A_f)$ the map obtained by filling the dotted arrow. Consider the diagram \eqref{diagramma}. We have just seen that the association $s \mapsto \rho(s)$ depends only on $\tau \in G_K^{ab}$, therefore also $\eta'(s) = \tau^*|_T \circ \rho(s^{-1})$ depends only on $\tau$. Thus $f(s)$ from Claim \ref{claim proof descendingI} depends only on $\tau$, and finally also $f_s = f_\tau$. To conclude that this is actually a Galois-descent data, we only need to check the cocycle condition. The proof is straightforward: for $\tau \in G_K$ and $X/ \overline{\Q}$ denote by $\tau_X$ the natural morphism of schemes $\tau_X \colon X^\tau \rightarrow X.$ Note that since $E \subset K$ the reflex field of $X$ and $X^\tau$ are identified, meaning that the same field $E$ acts both on $X$ and $X^\tau$. This implies that any isomorphism $f \colon X \rightarrow X^\tau$ is $E$-linear on the transcendental part of the cohomology. We extend the action of $E$ on the whole $H^2(X, \Q)(1)$ by letting it be the identity on $\NS(X)$ (and similarly we extend the action of $\A_{E,f}$ on the whole $H_{\et}^2(X, \A_f)(1)$). By construction, $f_\tau$ is the unique isomorphism $X^\tau \xrightarrow{\sim} X$ such that $\tau_X^* = f(\tau) \rho(\tau),$ where $\tau_X^* \colon H_{\et}^2(X, \A_f)(1) \rightarrow H_{\et}^2(X^\tau, \A_f)(1)$ is the map induced in cohomology.

\begin{claim}
For $\sigma, \tau \in G_K$ the isomorphism $ f_\sigma \circ f_\tau^\sigma \colon X^{\sigma \tau} \xrightarrow{\sim} X$ satisfies $$(\sigma \tau)_X^* \circ \rho(\sigma \tau)^{-1} = f_\sigma \circ f_\tau^\sigma.$$ In particular, $f_{\sigma \tau} = f_\sigma \circ f_\tau^\sigma$ by the unicity above. 
\end{claim}
\begin{claimproof}
From the commutative square 

\begin{center}
\begin{tikzcd}
X^\tau \arrow[r, "f_\tau"] & X   \\
X^{\sigma \tau} \arrow[u, "\sigma_{X^\tau}"]  \arrow[r, "f^\sigma_\tau"] & X^\sigma \arrow[u, "\sigma_X"]
\end{tikzcd}
\end{center}
one deduces that $(f^\sigma_{\tau})^* \colon H^2_\et(X^\sigma, \A_f)(1) \rightarrow H^2_\et(X^{\sigma \tau}, \A_f)(1)$ is given by $$(f^\sigma_{\tau})^* = \sigma_{X^{\tau}}^* \circ f(\tau) \circ {\sigma^{-1}_X}^*$$
so that
$$(f^\sigma_{\tau})^* \circ f^*_\sigma = \sigma(X^\tau)^* \circ f(\tau) \circ {\sigma^{-1}_X}^* \circ f(\sigma).$$
Writing $\sigma_X^* = f(\sigma) \rho(\sigma) $ we obtain 
$$(f^\sigma_{\tau})^* \circ f^*_\sigma = \sigma_{X^\tau}^* \circ f(\tau) \circ \rho(\sigma^{-1}).$$
Finally, we write $\sigma_{X^\tau}^* = f_{X^\tau}(\sigma) \rho_{X^\tau}(\sigma),$ with $f_{X^\tau}(\sigma)$ obtained as $f_\sigma$ but starting from $X^\tau$ instead that from $X$ itself. By the fact that the CM fields of $X$ and $X^\tau$ are naturally identified, we can identify $\rho_{X^\tau}(\sigma)$ with $\rho(\sigma)$, and by the linearity mentioned before stating the Claim, we conclude that 

\begin{equation} \label{cocycle}
(f^\sigma_{\tau})^* \circ f^*_\sigma =  f_{X^\tau}(\sigma) \circ f(\tau).    
\end{equation}

Consider now $(\sigma \tau)_X \colon X^{\sigma \tau} \rightarrow X$, and factorize it as $X^{\sigma \tau} \xrightarrow{\sigma_{X^\tau}} X^\tau \xrightarrow{\tau_X} X.$
It follows that $$(\sigma \tau)_X^* = f(\sigma \tau) \rho(\sigma \tau) = f_{X^\tau}(\sigma) \rho_{X^\tau}(\sigma) f(\tau) \rho(\tau).$$
Employing once again the linearity of the action of $E$ with respect to any Hodge isometry, we write 
$$ f_{X^\tau}(\sigma) \rho_{X^\tau}(\sigma) f(\tau) \rho(\tau) =  f_{X^\tau}(\sigma) f(\tau) \rho(\sigma \tau)$$
and we obtain 
$$ f_{\sigma \tau}^* = (\sigma \tau)_X^* \rho(\sigma \tau)^{-1} =   f_{X^\tau}(\sigma) f(\tau)  =^{\eqref{cocycle}} (f^\sigma_{\tau})^* \circ f^*_\sigma $$
and therefore 
$$f_{\sigma \tau} = f_\sigma f^\sigma_\tau .$$
\end{claimproof}
It follows that the assignment $\tau \mapsto f_\tau$ defines a Galois-descent data, and we employ this descent data to build the model $X_I$ of $X$ over $K$. Since by construction $f_{\tau}^{*} \mathcal{L} = \tau^* \mathcal{L}$ for every $\mathcal{L} \in \NS(X)$, we conclude that $G_K$ acts trivially on $\NS(\overline{X_I})$, i.e. $\rho(X_I) = \rho(X)$. In the same fashion, since $\tau^*|_T$ and $\eta'(\tau)$ agree modulo $I$, we have that $G_K$ acts trivially on $T(\overline{X}_I)[I]$ as well. The property is a direct consequence of Proposition \ref{descendingiso} and Corollary \ref{corollary}.
\end{proof}

\begin{rmk} \label{Galois rep canon model}
Similarly, one checks that the Galois representation $$ G_K \rightarrow \Aut(\widehat{T}(\overline{X}_I))$$ is given by $\rho \colon G_K \rightarrow U(\A_f).$
\end{rmk}
\begin{exs}
Let $X/\C$ be the Fermat quartic $x^4 + y^4 + w^4 + z^4 = 0$. Then $X$ has CM by $\Z[i]$, and with an appropriate choice of a basis, the transcendental lattice $T(X)$ can be represented by the quadratic form
\( \begin{bmatrix}
8 & 0 \\
0 & 8 \\
\end{bmatrix} \). One can show that the type of $X$ is $(\Z[i],4)$. Hence, the discriminant ideal of $X$ is $8 \Z[i]$, since the different ideal of $E$ is $2 \Z[i]$. Since $\rho(X) = 20,$ one can use the results in \ref{K3 class filds subsection} to determine the field $F_{\D_X}(E)/E$ easily. Using MAGMA, we found that  $F_{\D_X}(E)/E = E(\sqrt{2}) = \Q(\epsilon)$, where $\epsilon$ is a primitive eight root of unity. It is a classical fact that all the divisors of $V(x^4 + y^4 + w^4 + z^4) \subset \Pp^1_\Q$ are defined over $\Q(\epsilon),$ and we conclude that the canonical model of $X$ is nothing but the Fermat quartic over $\Q(\epsilon)$.
\end{exs} 
Since K3 surfaces with big discriminant can be descended canonically, we would like to understand how many there are. We start by considering principal K3 surfaces with complex multiplication by an imaginary quadratic field.
\begin{thm} \label{Classification20}
Let $X/\C$ be a principal K3 surface withs complex multiplication by an imaginary quadratic field $E$, so that $\rho(X) = 20$. Then $X$ has big discriminant unless 
\begin{itemize}
\item $E=\Q(i)$ and the type of $X$ is $(\Oo_E,1)$ (i.e., $T(X) \cong$ \( \begin{bmatrix}
2 & 0 \\
0 & 2 \\
\end{bmatrix} \)).
\item $E= \Q(\sqrt{-3})$ and the type of $X$ is  $(\Oo_E, 1)$ (i.e., $T(X) \cong$ \( \begin{bmatrix}
1 & 2 \\
2 & 1 \\
\end{bmatrix} \)).
\end{itemize}
\end{thm}
\begin{proof}
Let $X$ have complex multiplication by the ring of integers of $\Q(\sqrt{-d})$, with $d$ a square-free integer, and let $(I,\alpha)$ be the type of $X$. 
Suppose that $-d \equiv 2,3 \mod 4$. In this case, $\Oo_E = \Z[\sqrt{-d}]$ and $\D_E = (2  \sqrt{-d})$. Hence, having big discriminant means that the map 
\begin{equation} \label{BD}
\mu(E) \rightarrow \bigg( \frac{\Z[\sqrt{-d}] } {(\alpha) I \overline{I} (2  \sqrt{-d})} \bigg)^{\times}
\end{equation}
 is injective. If $d \neq -1$, then $\mu(E) = \{ \pm 1\}$ and the map   
$$\mu(E) \rightarrow \bigg( \frac{\Z[\sqrt{-d}] } {(2  \sqrt{-d})} \bigg)^{\times}$$
is already injective, so that we conclude thanks to Proposition \ref{keydiscriminant}. If $d=-1$, then $\mu(E) = \mu_4$ and the map 
\eqref{BD} has a kernel if and only if $(\alpha) \mathrm{Nm}_{E/\Q}(I) = \Z$. Since $\Z[i]$ is a UFD, every type $(I,\alpha)$ is equivalent to one of the form $(\Z[i], \alpha)$. Hence, the unique type in this case that has not big discriminant is $(\Z[i],1)$. \\
Suppose now that $-d \equiv 1 \mod 4$, so that $\D_E = (\sqrt{-d})$. If $d \neq 3$, then $\mu(E) = \mu_2$. Since $(2) \nsubseteq (\sqrt{-d})$, we conclude that
$$\mu_2 \rightarrow \bigg( \frac{\Oo_E } {(\sqrt{-d})} \bigg)^{\times}$$
has trivial kernel, hence $X$ has big discriminant. The last case left to consider is when $E= \Q(\sqrt{-3})$. Let $\omega := \frac{1+\sqrt{-d}}{2}$ be a primitive sixth-root of unity, so that $\Z[\omega]$ is the ring of integers of $E$. Since $\Z[\omega]$ is a UFD, we can suppose our type to be of the form $(\Z[\omega], \alpha)$ for some $\alpha \in \Q_{>0}$. The kernel of the map 
$$\mu_6 \rightarrow \bigg( \frac{\Z[\omega]} {(\sqrt{-3})} \bigg)^{\times}$$
is $\mu_3$, since $\omega^2 - 1 = \sqrt{-3} \omega$. Hence, $(\Z[\omega], \alpha)$ does not have big discriminant if and only if $\sqrt{-3} \omega \in (\alpha \sqrt{-3}),$ i.e. if and only if $\alpha =1$. 
\end{proof}
Therefore, there are exactly two (isomorphism classes of) complex K3 surfaces with CM by the ring of integers of an imaginary quadratic extension whose discriminant is not big. Coincidentally, these surfaces were studied in \cite{MR719348}. If the CM field is not quadratic imaginary, we have the following finiteness theorem.

\begin{thm} \label{almostall}
Let $E \subset \C$ be a CM number field, and denote by $K(E)$ the set of isomorphism classes of principal K3 surfaces over $\C$ whose reflex field equals $E$. Then, up to finitely many elements, every $X \in K(E)$ has big discriminant. 
\end{thm}

\begin{proof}
It is sufficient to prove that there are finitely many isomorphism classes of types without a big discriminant. Indeed, the type determines the transcendental lattice of a K3 surface, which in turn determines finitely many K3 surfaces (due to the finiteness of the Fourier-Mukai partners, see \cite{MR3586372} p. 373, Proposition 3.10). Let $\{ I_1, \cdots, I_n \}$ be the finite set of ideals for which the map $\mu(E) \rightarrow (\Oo_E / I_n)^\times$ is not injective and let $\{ J_1, \cdots, J_m \}$ be representatives of the classes of $\Cl(E)$. Every type $(J, \alpha')$ is equivalent to one of the form $(J_i , \alpha)$ for some $i \in \{1, \cdots, m \}$. Therefore, if $(J_i , \alpha)$ has not big discriminant, we have that $$(\alpha) J_i \overline{J_i} \D_E = I_j,$$
for some $j \in \{1, \cdots, n \}$. Fix now $i$ and $j$. We want to prove that there are only finitely many isomorphism classes of types of the form $(J_i, \alpha)$ such that the equality 
$$(\alpha) J_i \overline{J_i} \D_E = I_j$$
holds. To do this, suppose that both $(J_i, \alpha_1)$ and $(J_i, \alpha_2)$ have discriminant equals to $I_j$. In particular, we have that $(\alpha_1) = (\alpha_2)$, i.e., there exists a unit $u \in \Oo_E^\times$ such that $\alpha_1 = u \alpha_2$. Moreover, this unit is totally positive, since the signature of $T(X)$ does not depend on $X$. If we denote by $U$ the group of totally positive units, we see that the isomorphism type of $(J_i, u\alpha)$ for $u \in U$ depends only on the image of $u$ in the quotient $U / \mathrm{Nm}_{E/F}(\Oo_E^{\times})$, where $F$ denotes the maximal totally real subfield of $E$. Since the group $U / \mathrm{Nm}_{E/F}(\Oo_E^{\times})$ is finite, we conclude the proof. 
\end{proof}

\section{Some properties of the canonical models}
Let $X/\C$ be a complex K3 surface with CM by $\Oo_E$, and let $I \subset \D_X$ be an ideal as in Section \ref{descending section}. Consider the model $X_I$ constructed in Theorem \ref{DescendingI} over $K = F_I(E).$
\subsubsection{The $\ell$-adic Galois representation}
From the Remark \ref{Galois rep canon model} the Galois representation of $X_I$ is the map $\rho$ constructed in the proof of \ref{DescendingI}. In the next proposition we determine the Frobenius element of each prime of $K$ that does not divide $\ell$ and $I$. Let $\mathfrak{p} \subset \Oo_K$ be such prime, and write $\textrm{Nm}_{K/E}(\mathfrak{p}) = \mathfrak{q}^f$ with $\mathfrak{q} \subset \Oo_E$ a prime ideal and $f \geq 1$ the inertia degree of $\mathfrak{p}$.
We denote by $T_\ell(\overline{X}_I) \subset \mathrm{H}^2_\et(\overline{X}_I, \Z_\ell)(1)$ the $\ell-$adic transcendental lattice of $X_I$, and we have a natural decomposition $$\widehat{T}(\overline{X}_I) \cong \prod_{\ell \, \text{prime}}  T_\ell(\overline{X}_I).$$ We denote by $\rho_\ell \colon G_K \rightarrow \Aut( T_\ell(\overline{X}_I))$ the corresponding $\ell$-adic Galois representation.

\begin{prop} \label{Proposition Galois representation}
Let $F \subset E$ be the maximally totally real subfield, and assume that $\mathfrak{q}$ is coprime to $\ell$ and to $I.$ Then 

\begin{enumerate}
\item The Galois representation  $\mathrm{H}^2_\et(\overline{X}_I, \Z_\ell(1))$ is unramified at $\mathfrak{p};$
    \item If $\mathfrak{q}$ lies over an inert prime of $F$, then the Frobenius of $\mathfrak{q}$ acts as the identity on $\mathrm{H}^2_\et(\overline{X}_I, \Z_\ell(1))$;
    \item If $\mathfrak{q}$ lies over a prime of $F$ that splits in $E$, then there exists a unique $u \in E$ such that $u \overline{u} = 1$, $(u) = \big( \frac{\mathfrak{p}}{\overline{\mathfrak{p}}} \big)^f$ and $ u \equiv 1 \mod I.$ Then the Frobenius at $\mathfrak{p}$ acts on $T(\overline{X}_I)_\ell$ as multiplication by $u$ (and as the identity on $\NS(\overline{X}_I)).$
\end{enumerate}

\end{prop}

\begin{proof}
The Galois representation $\rho \colon G_K \rightarrow U(\A_f) \subset \Aut(\widehat{T}(\overline{X}_I))$ of $X_I$ factorizes through $G_K^{\ab}$, and the composition 

$$ \A_{K,f}^\times \xrightarrow{\art_K} G_K^{ab} \xrightarrow{\rho}  U(\A_f)$$

is the same map appearing in the commutative square \ref{square proof descending}. We call it $\rho$ also this time. We quickly recall how it is constructed. Let $\tau \in G_{K}^{ab}$ and let $t \in A_{K,f}^\times$ be such that $\art_{K}(t) = \tau$. Write $s = \mathrm{Nm}_{K/E}(t)$. Then $s \in S_I$, and there exists a unique $u \in E^\times$ such that $u \overline{u} =1,$ $u \frac{s}{\overline{s}} \in \widehat{\Oo}_E^{\times}$ and $u \frac{s}{\overline{s}} \equiv 1 \mod I$. The Galois representation is then given by $\rho(\tau) = u \frac{s}{\overline{s}}$. 
To study $\rho$ locally at $\mathfrak{p}$ we look at the restriction of $\rho$ at the decomposition group of $\mathfrak{p}$: $$\Gal(\overline{K}_{\mathfrak{p}} /K_{\mathfrak{p}}) \cong D_{\mathfrak{p}} \subset G_{k}^{ab},$$  which is well defined since we are only working with Abelian extensions. We have a commutative diagram

\begin{center}
\begin{tikzcd}

K_{\mathfrak{p}} \arrow[d, hook, "\iota"] \arrow[r,"\art_{\mathfrak{p}}"] & \Gal(\overline{K}_{\mathfrak{p}}/K_{\mathfrak{p}}) \arrow[d, hook] \\
\A_{K,f}^{\times}  \arrow[r, "\art_{K}"] & G_{K}^{ab}
\end{tikzcd} 
\end{center}
where $\art_{\mathfrak{p}}$ is the local Artin map. Note that if $t \in K_{\mathfrak{p}}$ then $\iota(t)$ is the id\`{e}le with value $x$ at the $\mathfrak{p}$-component and $1$ elsewhere. Since $\mathfrak{p}$ and $I$ are coprime, it follows that we automatically have $s \equiv 1 \mod I.$  Moreover, if $\ell$ and $\mathfrak{q}$ are coprime, the map $\widehat{T}(\overline{X}_I) \rightarrow \widehat{T}(\overline{X}_I)$ given by multiplication by $\iota(x)$ restricts to the identity on $T(\overline{X}_I)_\ell$.
 
\begin{enumerate}
    \item To prove that $\rho$ is unramified at $\mathfrak{p}$ we need to show that $\rho(I_\mathfrak{p}) = \{1\}$, where $I_\mathfrak{p} \subset G_K^{ab}$ is the inertia of $\mathfrak{p}$. Via the local Artin map, $I_\mathfrak{p}$ corresponds to $\Oo_{K,\mathfrak{p}}^{\times} \subset \A_{K,f}^{\times}.$ In particular, for $\tau \in I_{\mathfrak{p}} $ we can choose $t \in \im(\Oo_{K,\mathfrak{p}}^{\times} \rightarrow \A_{K,f}^{\times})$. It follows that $s \in \Oo_{K,\mathfrak{p}}^{\times},$ so that $u$ must be integral, and therefore $u \in \mu(E).$ But since $s \equiv 1 \mod I$ always, we see that the last condition translates to $u \equiv 1 \mod I,$ i.e., $u = 1.$ Hence $\rho(\tau)$ is given by multiplication by $\frac{s}{\overline{s}}.$ Finally, since $\ell$ and $\mathfrak{p}$ are coprime, $\frac{s}{\overline{s}}$ acts trivially on the $\ell$-adic transcendental lattice $T_\ell(\overline{X}_I)$. 
    
    \item Let $t \in \im(K_\mathfrak{p}^{\times} \rightarrow \A_{K,f}^{\times})$ be such that $t^{-1}$ is a local uniformizer. Then $\rho_\ell(\art_{K}(t))$ is the Frobenius at $\mathfrak{p}$. Since by assumption $\mathfrak{q}$ is inert over $F$, one concludes that $\frac{s}{\overline{s}} \in \widehat{\Oo}^\times_E$ in this case as well. Therefore one shows that $u = 1 $ too and one concludes as above. 
    
    \item This follows from the discussions at the beginning of the proof.

\end{enumerate}
\end{proof}

\subsubsection{The Picard group} \label{picard section} Let $X_I/K$ and $I$ be as in Theorem \ref{DescendingI}. In practice, it could be useful to know if every divisor of $X_I$ is defined over $K$. Note that this does not follow from the fact that the absolute Galois group of $K$ acts trivially on $\NS(\overline{X}_I) = \Pic(\overline{X}_I)$; in general one has a spectral sequence $$E^{p,q}_2:= \mathrm{H}^p(K, \mathrm{H}_{\et}^q(\overline{X} , \mathbb{G}_m) \Rightarrow \mathrm{H}_{\et}^{p+q}(X, \mathbb{G}_m)$$
which induces an exact sequence 
\begin{equation} \label{Picard sequence}
0 \rightarrow \Pic(X) \rightarrow \Pic(\overline{X})^{G_K} \rightarrow \Br(K) \rightarrow \Br(X).
\end{equation}
The quotient $\Pic(\overline{X})^{G_K} / \Pic(X) = \ker( \Br(K) \rightarrow \Br(X))$ is called the \textit{Amitsur group} of $X$ and it is denoted $\text{Am}(X)$. It is a finite Abelian group. 
\begin{defi}
Let $X$ an algebraic variety over a field $K$. The index of $X/K$ is $$\delta(X/K):= \gcd \{ [L:K] \colon [L:K]  < \infty \,\,\text{and} \,\, X(L) \neq \emptyset \}.$$
\end{defi}
\begin{prop} \label{index-Amitsur prop}
Let $X/K$ be a smooth projective and geometrically irreducible variety. Then $$\delta(X/K) \cdot \text{Am}(X) = \{0\}.$$ 
\end{prop}
\begin{proof}
This follows from the functoriality of \eqref{Picard sequence} and by a restriction-corestriction argument. 
\end{proof}
If $K$ is a number field for every place $v$ of $K$ consider the local index $\delta(X_{v} / K_{v})$ of the base change of $X$ to the completion $K_v$ of $K$ at $v$. 
\begin{cor}
If every local index of $X$ is one, the map $\Pic(X) \rightarrow \Pic(\overline{X})^{G_K}$ is an isomorphism. 
\end{cor}
\begin{proof}
This follows from Proposition \eqref{index-Amitsur prop} and the short exact sequence $$0 \rightarrow \Br(K) \rightarrow \bigoplus_v \Br(K_v) \xrightarrow{\sum_v \inv_v} \Q / \Z \rightarrow 0.$$
\end{proof}
In particular, if $X$ has a point everywhere locally, then the map $\Pic(X) \rightarrow \Pic(\overline{X})^{G_K}$ is an isomorphism. In the following, we collect some miscellaneous statements around these questions. 

\begin{prop} \label{proposition miscellanea}

Keeping the notation of above, we have 
\begin{enumerate}
\item Assume that $X_I$ contains a $(-2)$-curve. Then, there is a quadratic extension of $K$ over which all the line bundles of $X$ are defined.
  
    \item Suppose that $X_I$ has good reduction at a prime $\mathfrak{p}$ of $K.$ If $[E \colon \Q] \leq 12$ or $\mathrm{Nm}(\mathfrak{p}) \geq 18$ then $X_I$ has a local point at $\mathfrak{p}.$ 
   \item If $X_I$ is a Kummer surface, there exists a quadratic extension $K'/K$ and an Abelian surface $A/K'$ with $K'$-rational $2$-torsion such that $X_{I,K'} \cong \mathrm{Km}(A).$ It follows that $X_I(K') \neq \emptyset$ and $\Pic(\overline{X}_{I}) = \Pic(X_{I,K'})$.
\end{enumerate}
\end{prop}
\begin{proof}
\begin{enumerate}
   
\item Let $C \subset X$ be a $(-2)$-curve. By the fact that $C$ is unique in its linear system and that $G_K$ acts trivially on $\Pic(\overline{X}_I)$ we deduce that there exists a curve $C' \subset X_I$ whose base-change to $\C$ is $C$. It follows that $C'$ is a Severi-Brauer variety of dimension one over $K$, and hence that there is a quadratic extension $K'/K$ such that $C'_{K'} \cong \Pp^1_{K'}$.
    \item We use the notation of Proposition \ref{Proposition Galois representation}. Let $u \in E$ such that $u \overline{u} =1$ and the Frobenius at $\mathfrak{p}$ acts as multiplication by $u$ on $T_\ell(\overline{X}_I).$ By the Weil-conjecture it follows that the reduction of $X_I$ at $\mathfrak{p}$ has $q^2 + q \cdot \tr_{E/ \Q}(u) + q \cdot \rho(X) +1$ rational points. Since $u \overline{u} =1$, the bounds $-[E \colon \Q] \leq \tr_{E/ \Q}(u) \leq [E \colon \Q]$ hold. This, together with the fact that $\rho(X) = 22 - [E \colon \Q]$ and $[E \colon \Q] \leq 20$, implies point (3).
    \item In fact, as in point (1), all the sixteen exceptional divisors $E_1, \cdots, E_{16}$ are defined over $K$. Let $K'/K$ be a quadratic extension over which $E_{1,K'} \cong \Pp^1_{K'}$. Due to Corollary \ref{corollary} every isomorphism of $\overline{X}_{I}$ is defined over $K$, from which it follows that $\{E_1, \cdots, E_{16} \}$ are permuted by $\Aut(X_I)$. Hence, $E_{i, K'} \cong \Pp^1_{K'}$ for every $i = 1, \cdots, 16$.
    In $\Pic(X_{I, K'})$ there exists a reduced divisor $D \subset X_{I, K'}$ such that $2D = E_1 + \cdots + E_{16}$; one checks that  also $D$ must be defined over $K'$, because it is unique in its linear system. Let $\phi \colon Y \rightarrow X_{I, K'}$ be a $2$-covering associated to $D$. It follows that the ramification locus of $\phi$ can be written as $R_\phi := \sum_i C_i$, where each $C_i$ is a $(-1)$-curve isomorphic to $\Pp^1_{K'}$. From the arithmetic version of Castelnuovo's contractibility criterion (see \cite{MR1917232}, Theorem 3.7 p. 416) we can contract these curves to obtain a smooth surface $A$. Let $\pi \colon Y \rightarrow A$ denote the contraction map. We see that $A(K') \neq \emptyset$, since $\pi(C_i)$ is a $K'-$point, and the very same procedure carried out over $\C$ tells us that $A_{\C}$ is an Abelian surface such that $X \cong \text{Km}(A_\C)$. Therefore, $A$ is an Abelian surface too and $X_{I,K'} \cong \text{Km}(A)$. The fact that the full $2-$torsion is defined over $K'$ follows from the fact that every $2$-torsion point of $A$ can be written as $\phi(C_i)$ for some $i.$
\end{enumerate}

\end{proof}

\section{Applications}
In this last section, we derive the asymptotic \ref{theorem asymptotics introduction} and we study K3 surfaces with small fields of definition.

\subsection{Proof of Theorem \ref{theorem asymptotics introduction}} \label{subsection Proof of asymptotics}
In order to prove the estimates in Theorem \ref{theorem asymptotics introduction} we start with the following lemmas.

\begin{lemma} \label{lemma mult n}
Let $X/K$ be a K3 surface over a number field $K$ and let $n \in \Z$ be an integer. Denote by $B_n \subset \widehat{T}(\overline{X}) \otimes \Q / \Z $ the finite group given by 

$$B_{n} \coloneqq \big\{ \alpha \in \widehat{T}(\overline{X}) \otimes \Q / \Z \colon n \alpha \in  \big( \widehat{T}(\overline{X}) \otimes \Q / \Z \big)^{G_K} \big\}$$
and consider the natural Galois representation $\rho_n \colon G_K \rightarrow \Aut(B_n)$. Then, there exists a constant $C_n$ depending only on $n$ such that the $$[G_K \colon \ker(\rho_{n})] \leq C_n.$$
\end{lemma}

\begin{proof}
For any element $x \in \widehat{T}(\overline{X}) \otimes \Q / \Z$ we denote by $K(x)/K$ the field extension determined by $\{ \sigma \in G_K \colon \sigma^*(x) = x \}.$
Let $T(\overline{X})[n] :=  \widehat{T}(\overline{X}) \otimes \Z / n \subset  \widehat{T}(\overline{X}) \otimes \Q / \Z.$ Then $T(\overline{X})[n] \cong \big( \Z/ n \Z \big)^t$, and its Galois representation can be seen as a continuous morphism $G_K \rightarrow \Gl(22-\rho, \Z / n \Z).$ In particular, there exists a universal constant $C'_n := |\Gl(21, \Z/n \Z)| $ and an extension $K_1/K$ of degree bounded by $C'_n$ such that 
\begin{equation} \label{equation lemma}
    T(\overline{X})[n] \subset  \big( T(\overline{X}) \otimes \Q / \Z \big)^{G_{K_1}}.
\end{equation} 
Let $b \in \big( \widehat{T}(\overline{X}) \otimes \Q / \Z \big)^{G_K}$ and choose $b_n$ such that $nb_n = b$. If $K_2=K_1(b_n)$ it follows from \eqref{equation lemma} that $K_2/K_1$ is Galois, and that there is a natural inclusion $\Gal(K_2/K_1) \subset T(\overline{X})[n].$ Note that the Abelian group $\big( \widehat{T}(\overline{X}) \otimes \Q / \Z \big)^{G_K}$ has at most $21$ generators. Thus, putting $C_n \coloneqq C'_n \cdot |(\Z / n \Z )^{21}|  =  C'_n \cdot n^{21}$, we conclude the proof.
\end{proof}
\begin{lemma} \label{lemma for bound}
There exists a universal constant $A$ such that, if $X / \C$ is any K3 surface with CM by $\Oo_E$ for any $E$, and $K \subset \C$ is a field of definition for $X$ of minimal degree $F_X \coloneqq [K \colon \Q],$ then 
$$A \cdot [F_{3 \cdot \mathcal{D}_X}(E) \colon E] \leq F_X \leq 20 \cdot [F_{3 \cdot \mathcal{D}_X}(E) \colon E].$$
\end{lemma}

\begin{proof}
Let $C$ be the constant appearing in Remark \ref{remark main theorem introduction}, and consider an extension $K_1/K$ of degree bounded by $20 \cdot C$ such that $G_{K_1}$ acts trivially on $\NS(\overline{X})$ and $E \subset K_1.$ Note that in particular
$$ T(\overline{X})[\D_X] \subset \big( \widehat{T}(\overline{X}) \otimes \Q / \Z \big)^{G_{K_1}},$$
which implies that $F_{\D_X}(E) \subset K_1$. Let $\rho_3 \colon G_{K_1} \rightarrow \Aut(B_3)$ be the Galois representation from Lemma \ref{lemma mult n} with $n=3$, and let $K_2/K_1$ be the extension associated to $\ker(\rho_3)$. It follows by construction that 
$$ T(\overline{X})[3 \cdot \D_X] \subset \big( \widehat{T}(\overline{X}) \otimes \Q / \Z \big)^{{G_K}_2},$$
and moreover by the Remark \ref{remark (3)}, the ideal $I = 3 \cdot \D_X$ has the property that $\mu(E) \rightarrow (\Oo_E / I)^\times$ is injective. By the property of Theorem \ref{main theorem introduction}, we conclude that $F_{3 \cdot \D_X} \subset K_2,$ and we can pick $A= \frac{1}{C \cdot C_3 \cdot 20}$. The second inequality follows since $F_{3 \cdot \D_X}$ is a field of definition for $X$.
\end{proof}

Thus, it remains to study the degrees $[F_{3 \cdot \mathcal{D}_X}(E) \colon E].$ We introduce the following notation (some of which was already introduced in \ref{K3 class filds subsection}).

\begin{itemize}
\item $F \subset E$ is the maximal totally real subfield of $E$;
    \item For any ideal $I \subset \Oo_E$, we denote by $E^{I,1} = \{ e \in E^\times \colon \ord_\mathfrak{p} (e-1) \geq \ord_\mathfrak{p} I \,\, \forall \,\, \mathfrak{p} | I \}$.
    \item $\Oo_E^{I}:= \Oo_E^\times \cap E^{I,1}$;
    \item We put $d_X := \D_X^{G} = \D_X \cap F$;
    \item Let $e(d_X)$ be the product of the ramification indices in $E$ of all the places of $F$ that are coprime to the ideal $d_X \subset \Oo_F$ (in particular, all the infinite places are taken into account);
    \item Finally, let $G \subset \Aut(E)$ be the subgroup generated by the complex conjugation. For any $G-$module $N$ we write $\mathrm{H}^i(M) := \mathrm{H}^i(G,M)$.
\end{itemize}
In Chapter 10 of \cite{VALLONI2021107772} we proved that

\begin{equation} \label{cardinality class field}
[F_{\D_X}(E) \colon E] =  \frac{2 \cdot h_E \cdot \phi_E(\D_X) \cdot [\Oo^{\times}_F \colon \mathrm{Nm}_{E/F}(\Oo^{\D_X}_E)]}{h_F \cdot \phi_F(d_X) \cdot [\Oo^{\times}_E \colon \Oo^{\D_X}_E] \cdot e(d_X) \cdot | \mathrm{H}^1(E^{\D_X,1}) |},
\end{equation}
\begin{prop}
There are constants $B_1, B_2 > 0$ such that 
$$ 
B_1 \leq \frac{2 \cdot [\Oo^{\times}_F \colon \mathrm{Nm}_{E/F}(\Oo^{3 \cdot \D_X}_E)]}{[\Oo^{\times}_E \colon \Oo^{3 \cdot\D_X}_E] \cdot e(E/F, 3 \cdot d_X) \cdot | \mathrm{H}^1(E^{3 \cdot \D_X,1}) |} \leq B_2 $$
for any K3 surface $X/ \C$ with CM by any $\Oo_E.$
\end{prop}

\begin{proof}
In the following short exact sequence, we note that the first vertical arrow is surjective
 \begin{center}
\begin{tikzcd} 
 1 \arrow[r] 
& \Oo_E^{3 \cdot \D_X} \arrow[r] \arrow[d] 
& \Oo_E^\times \arrow[r] \arrow[d, "\mathrm{Nm}_{E/F}"] 
& \Oo_E^\times / \Oo_E^{3 \cdot \D_X} \arrow[r] \arrow[d]
& 1 \\
1 \arrow[r] 
& \mathrm{Nm}_{E/F}(\Oo^{3 \cdot \D_X}_E) \arrow[r] 
& \Oo_F^{\times} \arrow[r] 
& \Oo_F^{\times} / \mathrm{Nm}_{E/F}(\Oo^{3 \cdot \D_X}_E)  \arrow[r]
& 1.
\end{tikzcd}
\end{center}
Since $\mu(E) \rightarrow (\Oo_E / 3\D_X )^\times $ is injective we obtain the following exact sequence
$$ 1 \rightarrow \mu(E) \rightarrow \Oo_E^\times / \Oo_E^{\D_X} \rightarrow  \Oo_F^{\times} / \mathrm{Nm}_{E/F}(\Oo^{3 \cdot\D_X}_E)  \rightarrow \Oo_F^{\times} / \mathrm{Nm}_{E/F}(\Oo^{\times}_E) \rightarrow 1,$$
that implies 
\begin{equation} 
\frac{ [\Oo^{\times}_F \colon \mathrm{Nm}_{E/F}(\Oo^{3 \cdot \D_X}_E)]}{[\Oo^{\times}_E \colon \Oo^{3 \cdot \D_X}_E]} = \frac{[ \Oo_F^{\times} \colon \mathrm{Nm}_{E/F}(\Oo^{\times}_E)  ]}{|\mu(E)|}.
\end{equation} 
Moreover, $$1 \leq [ \Oo_F^{\times} \colon \mathrm{Nm}_{E/F}(\Oo^{\times}_E)  ]   = 2 \cdot  [\Oo_F^\times \colon \Oo_F^{\times 2}] = 2^{[F \colon \Q]}$$
and $$1 \leq |\mu(E)|  \leq M,$$
where $M$ is the biggest integer such that $\phi(M) \leq 20.$
Concerning the term $e(3 \cdot d_X),$ by Proposition \ref{keydiscriminant} we have that $\D_X \subset (2)^{-1} \D_{E/F}$. Since the primes the divide $\D_{E/F} \cap F$ are exactly the ones that ramify in $E$, we see that in the product $e(3 \cdot d_X)$, only the places at infinity appear, and at most the places over $2$, so that $e(3 \cdot d_X) \leq 2^{2 [F \colon \Q]} \leq 2^{20} .$ 
Finally, the term $\mathrm{H}^1(E^{3 \cdot \D_X,1})$ is described in \cite[Proposition 11.5]{VALLONI2021107772}, and it also can be universally bounded for any CM field $E$ with $[E \colon \Q] \leq 20.$ 
\end{proof}
This, together with Lemma \ref{lemma for bound}, proves the existence of constants $A_1, A_2$ which are independent on $X$ and on $E$ and such that  

$$A_1 \cdot \frac{\phi_E(\D_X)}{\phi_F(d_X)} \cdot \frac{h_E}{h_F} \leq F_X \leq A_2 \cdot \frac{\phi_E(\D_X)}{\phi_F(d_X)} \cdot \frac{h_E}{h_F}$$
for any K3 surface $X$ with CM by any $\Oo_E.$ Hence, Theorem \ref{theorem asymptotics introduction} is proved.

Note that the quantity $\frac{h_E}{h_F}$ only depends on the rational Hodge structure $T(X)_\Q$, while $\frac{\phi_E(\D_X)}{\phi_F(d_X)}$ depends on the integral Hodge structure $T(X)$. If $E$ is fixed then $\frac{\phi_E(\D_X)}{\phi_F(d_X)}$ is unbounded as we let $|\disc(T(X))|$ grow. This shows the dependence between the degree of a minimal field of definition for $X$ and the discriminant of its N\'{e}ron-Severi group. 
On the other hand, it is worthwhile to point out that if $E$ is allowed to vary, then the magnitude of $|\disc(T(X))|$ alone does not ensure that $\frac{\phi_E(\D_X)}{\phi_F(d_X)}$ is big. Consider for example $d < 0$ any square-free integer such that $d \equiv 1 \mod 4.$ Let $\Oo_d$ be the ring of integers of $E_d = \Q(\sqrt{d})$ and let $X_d$ be the unique K3 surface of maximal Picard rank whose transcendental lattice is isometric to $\Oo_d$ with the quadratic form given by $(x,y) \mapsto \tr (x \overline{y}).$ Then the discriminant ideal of $X_d$ corresponds to the different $\D_{E_d}$ of $\Oo_d$, and in particular $\frac{\phi_E(\D_X)}{\phi_F(d_X)} =1$ in this case. This can be generalized to every CM field $E$ and it shows that in order to prove Theorem \ref{finiteness}, one also needs the finiteness of CM fields $E$ satisfying $\frac{h_E}{h_F} \leq N$ for any given $N>0$. This finiteness follows from Lemma 4 of \cite{zbMATH04077384}, which is a corollary of the Brauer-Siegel theorem:  
\begin{prop}
Given $N>0$, there are only finitely many CM number fields $E$ of a given degree that satisfy $h_E/h_F \leq N$. 
\end{prop}
To conclude the proof of Theorem \ref{finiteness}, we use an argument similar to the one in \ref{almostall} to deduce that for each CM field $E$, there are only finitely many K3 surfaces with $\frac{\phi_E(\D_X)}{\phi_F(d_X)} \leq N.$

\subsection{Small fields of definition.} \label{subsection Principal CM K3 surfaces with small fields of definition}
After the work of Sch\"{u}tt mentioned in the introduction, we would like to study K3 surfaces that admit small or natural fields of definition. We apply our results to classify K3 surfaces with CM by $\Oo_E$, where $E$ is quadratic imaginary, that admit a model of full Picard rank over $E$ or over the Hilbert class field of $E$, which we denote $K(E)$. 

Using the notation introduced in \ref{K3 class filds subsection}, let $\Cl_I(E)$ be the ray class group of $E$ modulo $I$. One has an exact sequence 
\begin{equation} \label{fundamental exact sequence}
    1 \rightarrow \Oo_E^I \rightarrow \Oo_E^\times \rightarrow (\Oo_E/I)^\times \rightarrow \Cl_I(E) \xrightarrow{\pi} \Cl(E) \rightarrow 0
\end{equation}
\begin{thm} \label{theorem small field of defi singular}
Let $X$ be a K3 surfaces with CM by $\Oo_E,$ and assume that $E$ is a quadratic imaginary field. If $X$ has big discriminant, then
\begin{enumerate}
    \item $X$ admits a model of Picard rank $20$ over $E$ if and only if $G$ acts trivially on $\Cl_{\D_X}(E),$ the ray class group of $E$ modulo $I;$
    \item $X$ admits a model of Picard rank $20$ over $K(E)$ if and only if $G$ acts trivially on $(\Oo_E/ \D_X)^\times / \mu(E).$
    \item Let $\phi$ be the Euler function and assume that $2$ does not ramify in $E$. Then $[F_{\D_X}(E) \colon E] = h_E \frac{\phi_E(\D_X)}{\phi(d_X)} $, and we have 
    \begin{itemize}
        \item $(1)$ happens if and only if $h_E = \frac{\phi_E(\D_X)}{\phi(d_X)} = 1$;
        \item $(2)$ happens if and only if $ \D_X = \D_E. $
    \end{itemize}
\end{enumerate}
\end{thm}

\begin{proof}

Note that $E \subset K(E), F_{\D_X}(E) \subset  K_{\D_X}(E)$.
The Galois group of $K_{\D_X}(E)/E$ isomorphic to $\Cl_{\D_X}(E)$, and via Galois theory $K(E)$ corresponds to $\ker(\pi)$ where $\pi$ is as in \ref{fundamental exact sequence} and $\F_{\D_X}(E)$ corresponds to $\Cl_{\D_X}(E)^G$ where $G \cong \Z/ 2\Z$ is the group generated by the complex conjugation.  
\begin{enumerate}
    \item Point (1) is equivalent to  $ F_{\D_X}(E) =E,$ i.e. to the following inclusion $$\Cl_{\D_X}(E)^G =  \Cl_{\D_X}(E).$$.
    \item To prove (2), note that $F_{\D_X}(E) \subset K(E)$ if and only if $(\Oo_E/ \D_X)^\times / \mu(E) \cong\Gal(K_{\D_X}(E)/K(E)) \subset \Gal(K_{\D_X}(E)/F_{\D_X}(E)) \cong \Cl_{\D_X}(E)^G.$
    \item Consider the formula
    $$
[F_{\D_X}(E) \colon E] =  \frac{2 \cdot h_E \cdot \phi_E(\D_X) \cdot [\Oo^{\times}_F \colon \mathrm{Nm}_{E/F}(\Oo^{\D_X}_E)]}{h_F \cdot \phi_F(d_X) \cdot [\Oo^{\times}_E \colon \Oo^{\D_X}_E] \cdot e(d_X) \cdot | \mathrm{H}^1(E^{\D_X,1}) |}.$$
Since $E$ is quadratic imaginary we know by Proposition \ref{keydiscriminant} that $\D_X \subset \D_E$, so that $e(d_X) = 2$. Moreover, under the assumption that $2$ does not ramify we have that $ \mathrm{H}^1(E^{\D_X,1}) =0.$ Finally, since $\mu(E) = \{ \pm 1 \}$ and $X$ has big discriminant, we conclude that $[F_{\D_X}(E) \colon E]  = h_E \cdot \frac{\phi_E(\D_X)}{\phi(d_X)},$ so that (1) happens if and only if $h_E =1$ and $\frac{\phi_E(\D_X)}{\phi(d_X)}=1.$ Finally, the last point follows from by forcing the equality $[F_{\D_X}(E) \colon E]  = h_E$ which becomes $ 1 = \frac{\phi_E(\D_X)}{\phi(d_X)}$. Using that $\D_X \subset \D_E$ we readily conclude that $ \D_X = \D_E$.
\end{enumerate}
\end{proof}

\begin{exs}
We consider a fundamental discriminant of class number one: $$d \in \{-7, -8, -11, -19, -43, -67, -163 \}.$$ For sake of simplicity, we do not consider $d= -3$ or $d= -4$. For any such $d$, let $\Oo_d$ be the ring of integers of $E_d := \Q(\sqrt{d})$. Let $X_d / \C$ be the unique K3 surfaces of type $(\Oo_d, 1)$. Its discriminant ideal is $\D_X = \D_{E_d}$, and if $d$ is odd then $\D_{E_d}= (\sqrt{d})$ whereas if $d= -8$ we have $\D_{E_d} = (2 \sqrt{-2})$. Since $\mu(E_d) = \{ \pm 1 \}$, we readily check that $\mu(E_d) \rightarrow \Oo_d / \D_d$ is injective for every such a $d$. Therefore, $X_d$ admits a canonical model $X_d^{\text{can}}$ over $F_{K3, \D_d}(E_d)$. The proposition above shows that $F_{K3, \D_d}(E_d) = E_d$ whenever $d$ is odd, and it is possible to show that the same hold for $d= -8.$ Therefore, $X_d^{\text{can}}$ can be defined over the CM field $E_d$. Elkies in the webpage \cite{ElkiesWeb} listed Weierstrass equations for all the K3 surfaces over $\Q$ of maximal Picard rank, with disciminant $d$ and N\'{e}ron-Severi defined over $\Q$. By the  property in Theorem \ref{bigggg}, these are our canonical models (once base-changed to $E_d$). Therefore, we have a list of explicit equations: 
\begin{itemize}
\item $X^{\text{can}}_{-7} \,\colon y^2 = x^3 - 75x - (64t + 378 + 64/t)$;
\item $X^{\text{can}}_{-8} \,\colon y^2 = x^3 - 675x + 27(27t - 196 + 27/t)$;
\item $X^{\text{can}}_{-11} \,\colon y^2 = x^3 - 1728x - 27(27t + 1078 + 27/t)$; 
\item $X^{\text{can}}_{-19} \,\colon y^2 = x^3 - 192x - (t + 1026 + 1/t)$;
\item $X^{\text{can}}_{-43} \,\colon y^2 = x^3 - 19200x - (t + 1024002 + 1/t)$;
\item $X^{\text{can}}_{-67} \,\colon y^2 = x^3 - 580800x - (t + 170368002 + 1/t)$;
\item $X^{\text{can}}_{-163} \,\colon y^2 = x^3 - 8541868800x - (t + 303862746112002 + 1/t)$.

\end{itemize}

\end{exs}

Note that the part of the statement concerning the Hilbert class field cannot be generalized to CM fields $E$ of higher degree, as $F_I(E)$ is not usually contained in $K_I(E)$ anymore. On the other hand, one can still determine which K3 surfaces can be defined over $E$, and the next proposition is the generalization of the one above in this sense. We omit the proof as it is the same to the one above, thanks to the facts listed in \ref{K3 class filds subsection}. 
\begin{prop} \label{definedE}
Let $X / \C$ by a K3 surface with CM by the ring of integers of a CM number field $E$, and assume that $X$ has big discriminant. Then $X$ admits a model with full Picard rank over $E$ if and only if
\begin{enumerate}
\item The complex conjugation acts trivially on $\Cl_{\D_X}(E)$ and
\item The natural inclusion $\mathrm{Nm}_{E/F}(\Oo_E^{\D_X}) \subset \Oo_F^\times \cap \mathrm{Nm}_{E/F}(E^{\D_X,1})$ is an isomorphism.
\end{enumerate}
If each prime of $F$ over $2$ does not ramify in $E$, this is equivalent to
$$\frac{h_E}{h_F} = \frac{\phi_{E}(\D_X)}{\phi_{F}(d_X)} = 1.$$
\end{prop}

\subsection{K3 class fields and ring class fields}
Let $X/ \C$ be a K3 surface of maximal Picard rank and let $d \coloneqq \disc(T(X)).$ Let $H(d)$ be the ring class field of $d$. The following results are due to Sch\"{u}tt.

\begin{thm}{\cite[Theorem 2]{MR2602669}}
Let $L$ be a number field over which $X$ admits a model with every divisor defined over $L$. Then $H(d) \subset L(\sqrt{d}).$
\end{thm}

\begin{prop}{\cite[Proposition 10]{MR2346573}}
A K3 surface $X$ of Picard rank $20$ always admits a model over $H(d).$
\end{prop}

Later, Hulek and Sch\"{u}tt in \cite{MR2950152} analyzed the Galois action on $\NS(\overline{X})$, where $X/H(d)$ is the model from the proposition above, and concluded that every line bundle of $X$ is defined over $H(d)$. Thus, if $X$ is also principal, the property in Theorem \ref{DescendingI} allow us to conclude that $F_{\D_X}(E) = H(d),$ a fact that we missed to notice in our previous work. 

We conclude by giving an independent proof of the equality $F_{\D_X}(E) = H(d)$ when $d$ is a principal discriminant, i.e., $d$ is the discriminant of some quadratic imaginary field $E$. In this case, $H(d)$ is the Hilbert class field of $E$, which we still denote by $K(E)$. We also assume that $\mu(E) = \{ \pm 1 \}$, the general case follows with some minor adjustments, taking into consideration Theorem \ref{Classification20}.  

\begin{prop}
Let $X$ be as before, and suppose that $\disc(\NS(X)) = \disc(\Oo_E).$ Then $K(E) = F_{\D_X}(E).$
\end{prop}

\begin{proof}
First, one notes that $\D_X = \D_E$ since $\D_X \subset \D_E$ by \ref{keydiscriminant} and $$\disc(\NS(X)) = \mathrm{Nm}(\D_X) = |\disc(E)| = \mathrm{Nm}(\D_E).$$ Thus, we only need to show that $F_{\D_E}(E) = K(E)$. For simplicity, we write $\D \coloneqq \D_E = \D_X$. Consider again the exact sequence \eqref{fundamental exact sequence} with $I = \D$:

\begin{equation}
    1 \rightarrow \Oo_E^\D \rightarrow \Oo_E^\times \rightarrow (\Oo_E/ \D)^\times \rightarrow \Cl_{\D}(E) \xrightarrow{\pi} \Cl(E) \rightarrow 0
\end{equation}

As proved in \ref{theorem small field of defi singular}, it is enough to show that $(\Oo_E/ \D)^\times/\{\pm 1\} = \Cl_{\D}^G, $ where $G$ denotes the group generated by complex conjugation. Since $G$ acts trivially on $(\Oo_E/ \D)^\times/\{\pm 1\}$, one inclusion is for free. To prove the other, denote by $\tilde{\Cl}(E) \coloneqq \pi^{-1}(\Cl(E)[2])$, where $\Cl(E)[2] = \cl(E)^G.$ Then $\Cl_{\D}^G \subset \tilde{\Cl}(E);$ taking invariants in the sequence above we obtain
$$1 \rightarrow (\Oo_E/ \D)^\times/\{\pm 1\} \rightarrow \Cl_{\D}^G \rightarrow \Cl(E)[2]  \xrightarrow{\delta} H^1(G,(\Oo_E/ \D)^\times/\{\pm 1\}).$$
Thus $K(E) = F_{\D}(E)$ if and only if $\delta$ is injective. Since $G$ acts trivially on $(\Oo_E/ \D)^\times/\{\pm 1\}$ we can identify $H^1(G,(\Oo_E/ \D)^\times/\{\pm 1\})$ with the $2$-torsion of $(\Oo_E/ \D)^\times/\{\pm 1\}$. Moreover, $\delta$ is described as follows: for any $i \in \Cl(E)[2] $ let $j \in \tilde{\Cl}(E)$ such that $\pi(j) = i. $ Then $j/ \overline{j} \in (\Oo_E/ \D)^\times/\{\pm 1\}$ by \eqref{fundamental exact sequence} and $\delta(i) = j/ \overline{j}.$
The construction above also shows that we can consider $\delta$ as a map $\delta \colon \tilde{\Cl}(E) \rightarrow (\Oo_E/ \D)^\times/ \{\pm 1 \}.$

\begin{claim}
There is a natural lift $\tilde{\delta} \colon \tilde{\Cl}(E) \rightarrow (\Oo_E/ \D)^\times$ such that

\begin{center}
\begin{tikzcd}
& (\Oo_E/ \D)^\times  \arrow[d]  \\
\tilde{\Cl}(E) \arrow[r, "\delta" ] \arrow[ur, "\tilde{\delta}"] & (\Oo_E/ \D)^\times/\{\pm 1\},
\end{tikzcd} 
\end{center}
commutes. 
\end{claim}
Before proving the claim, we introduce the following notation. Let $I \subset \Oo_E$ be any ideal. 
\begin{itemize}
    \item $\mathcal{I}^{I}$ denotes the group of fractional ideals of $E$ that are coprime to $I$;
    \item $E^{1, I} \subset E^\times$ is as in \ref{K3 class filds subsection}, and $\mathcal{P}^I \subset \mathcal{I}^{I}$ is the subgroup generated by the elements of $E^{1, I}$. 
    \item $E^{I} \coloneqq \{ e \in E^\times \colon (e) \in \mathcal{I}^{I} \}.$
\end{itemize}
In this way, we have an isomorphism $\Cl_{I} =  \mathcal{I}^{I} / \mathcal{P}^I.$

\begin{claimproof}
 If $i \in \Cl(E)[2]$ then there are ramified primes $ p_1, p_2, \cdots p_r \in \Z_{\geq 0}$ such that, if $\mathfrak{p}_1$ denotes the unique prime ideal of $E$ over $p_i,$ then the ideal $\mathfrak{P} \coloneqq \prod_i \mathfrak{p}_i$ represents the ideal class $i$ in $\Cl(E)[2] \subset \Cl(E)$. Let $J \in \mathcal{I}^{\D}$ be any ideal that represents an element $j \in \tilde{\Cl}_{\D}(E)$ that lies over $i \in \Cl(E)[2]$ via the natural projection $\tilde{\Cl}_{\D}(E) \rightarrow \Cl(E)[2]$. Then $J$ can be uniquely written as $(e) / \mathfrak{P}$ for some element $e \in E^\times$ uniquely determined up to $\pm 1.$ The class of $e/ \overline{e}$ in $E^{ \mathcal{D}}/ E^{1,\mathcal{D}} \cong (\Oo_E/\D)^\times$ only depends on $j,$ so that  $$\tilde{\delta}(j) \coloneqq e/ \overline{e}  \in (\Oo_E/\D)^\times$$
is the lift we are after.
\end{claimproof}
We are going to prove that $\ker(\tilde{\delta}) = (\Oo_E/\D)^\times/\{ \pm 1\}.$ Write $E = \Q(\sqrt{d})$ with $d < 0$ a square-free integer and put $\mathcal{D} \coloneqq \D_E$ and $D \coloneqq \mathrm{Nm}(\D).$ In the following we shall make use of the fact that if $v$ is a non-archimedean, normalized valuation of $E$ that is ramified over $\Q$, then $v(x) \in 2 \Z$ for any $x \in \Q^\times.$ As it often happens with imaginary quadratic fields, we need to separate the cases depending on $d$.

\begin{itemize}
    \item Suppose that $d \equiv 1 \mod 4$, so that $2$ does not ramify in $E$. Let $p_1, \cdots, p_n\in \Z_{\geq 0}$ be the finite primes of $\Q$ that ramify in $E$, and suppose that $i \neq 0.$ Then as explained in the proof before we can assume without loss of generality that $i = [\mathfrak{P}]=[\mathfrak{P}^{-1}]$ where $\mathfrak{P} = \mathfrak{p}_1 \cdots \mathfrak{p_r}$ for some $0<r<n$. Let $e \in E^{\times}$ be such that the ideal $(e) \mathfrak{P}^{-1}$ belongs to $\mathcal{I}^{\D}.$ We only need to show that $e / \overline{e} \neq 1 \mod \D.$ Suppose by contradiction that 
\begin{equation} \label{equation proof ring class}
e = \overline{e} + \overline{e} \Delta
\end{equation}
for some $\Delta \in \D \cdot E^{\D}.$ Write $e = x + y \sqrt{d}$ with $x,y \in \Q.$ 
The fact that $v(x) \in 2\Z$ for any $x \in \Q^{\times}$ implies that $y \neq 0,$ because otherwise one could not have that $(e) \mathfrak{P}^{-1} \in \mathcal{I}^{\D}$. Similarly, since $i \neq 0$ by assumption, one shows that $x \neq 0$ too. Let $v_i$ the valuation associated to $\mathfrak{p}_i$. Then $v_i(\sqrt{d}) = 1$ for any $i \in \{1, \cdots, r \}.$ Using the non-archimedean property, we have $v_i(x + y \sqrt{d}) = \min \{v_i(x), v_i (y \sqrt{d}) \},$ and since $(e) \mathfrak{P}^{-1} \in \mathcal{I}^{\D}$ then $\min \{v_i(x), v_i (y \sqrt{d}) \} = 1$. So we conclude that $v_i(y) = 0$ and that $v_i(x) > 0.$ Since $x$ is rational this means that $x \equiv 0 \mod p_i$ for any $i \in \{1, \cdots, r \}.$ Substituting in \eqref{equation proof ring class} we obtain $$x + y \sqrt{d} = (x - y \sqrt{d})(1+ \Delta),$$
hence 
$$2 y \sqrt{d} = x \Delta - y \sqrt{d} \Delta.$$

We pick an $i = 1, \cdots, r$ and we reduce the equation above modulo $\mathfrak{p}_i^2 = (p_i)$ to obtain 
$$2 y \sqrt{d} \equiv 0 \mod p_i,$$
i.e., $y \equiv 0 \mod \mathfrak{p}_i$ since $2$ and $\mathfrak{p}_i$ are coprime. This contradicts $v_i(y) = 0$ for each $i \in \{1, \cdots, r \}.$
    
\item If $d \equiv 3 \mod 4$ then $2$ ramifies in $E$, $D = 4d$ and $\mathcal{D} = 2 (\sqrt{d}).$ If $2$ does not appear among $p_1, \cdots, p_r$ then we proceed as above. Otherwise, we can assume without loss of generality that $p_1 = 2.$ Then reasoning as before, since $(x+y\sqrt{d}) \mathfrak{P}^{-1} \in \mathcal{I}^{(2 \sqrt{d})}$ we must have that $1= v_1(x+y \sqrt{d}) \geq \min \{v_1(x), v_1 (y) \} $ because $2$ and $d$ are coprime. Since both $x$ and $y$ are rational, a quick analysis shows that this can only happen if $v_1(x) = v_1(y) = 0$ (and therefore $x \equiv y \mod 2).$ Expanding the equality $x+ y \sqrt{d} = (x - y \sqrt{d})(1+ \Delta)$ we get $$2y \sqrt{d} = x \Delta + y \sqrt{d} \Delta.$$ Note that $\Delta \in (2 \sqrt{d}) \in \Oo_E$, so that we can divide both sides by $2$ and obtain 
$$y \sqrt{d} = x \Delta' + y \sqrt{d} \Delta'.$$
Using that $x \equiv y \equiv \sqrt{d} \equiv 1 \mod \mathfrak{p}_1$ we see that the equation above forces $y  \equiv 2 \Delta' \equiv 0 \mod \mathfrak{p}_1,$ another contradiction. 
    
    \item Finally, if $d \equiv 2 \mod 4$ then $D \equiv 0 \mod 8$ since $\D = (2 \sqrt{d})$. If all the $p_1, \cdots, p_r$ are odd, then we conclude as in the first point. Otherwise, we put $p_1 =2$ and we note that this time the fact that $$(x+ y\sqrt{d})/ \mathfrak{P} \in \mathcal{I}^{\D}$$ implies that $v_1(y) = 0$ and $v_1(x)>0$. As above, assume that there is $\Delta \in \D \cdot E^{\D}$ such that the equation $$ 2y \sqrt{d} = \Delta x - y \sqrt{d} \Delta $$ holds. Computing the valuations, we see that $v_1(2y \sqrt{d}) = 3,$ $v_1(\Delta x) = 3 + v_1(x)$ with $v_1(x) \in 2 \Z_{>0}$ since $x \in \Q$ and $v_1(y \sqrt{d} \Delta) = 4.$ But this can happen only if $v_1(x)=0,$ a contradiction.

\end{itemize} 

\end{proof}

\bibliography{bibliographyT}
\bibliographystyle{alpha}

\end{document}